\documentclass[a4paper,abstracton]{scrartcl}
\usepackage{lmodern}
\usepackage[T1]{fontenc}
\usepackage[utf8]{inputenc}
\usepackage{mathtools}
\usepackage{etoolbox}
\usepackage{amsmath}
\usepackage{amsthm}
\usepackage{amssymb}
\usepackage{geometry}
\usepackage{enumitem}
\usepackage{mathrsfs}
\usepackage{cleveref}
\usepackage[ mathscr ]{euscript}
\usepackage[subnum]{cases}
\usepackage{empheq}
\usepackage{esint} %\fint
\usepackage{dsfont} %IdentityMatrix
\usepackage{tikz-cd}
\usepackage{graphicx}
\usepackage{todonotes}
\usepackage[toc,page]{appendix}

\usepackage{tcolorbox}
\tcbuselibrary{theorems}
\newtheoremstyle{break}% name
  {}%         Space above, empty = `usual value'
  {}%         Space below
  {\itshape}% Body font
  {}%         Indent amount (empty = no indent, \parindent = para indent)
  {\bfseries}% Thm head font
  {.}%        Punctuation after thm head
  {\newline}% Space after thm head: \newline = linebreak
  {}%         Thm head spec

\newtheorem{thm}{Theorem} [section]
\newtheorem{lemma}[thm]{Lemma}

\newtheorem{prop}[thm]{Proposition}

\theoremstyle{remark}
\newtheorem*{remark}{Remark}

\theoremstyle{definition}

\newtcbtheorem[number within=section]{question}{Question}
{colback=black!5,colframe=black!90,fonttitle=\bfseries}{th}

\DeclareMathOperator*{\esssup}{ess\,sup}

\DeclareMathOperator{\id}{id}
\DeclareMathOperator{\tr}{tr}

\DeclareMathOperator{\dist}{dist}
\DeclareMathOperator{\SO}{SO}

\DeclareMathOperator{\bdy}{bdy}

\DeclareMathOperator*{\argmin}{arg\,\min}

\newcommand{\mr}{\mathbb R}
\newcommand{\R}{\mathbb R}
\newcommand{\td}{\mathrm{d}}

\newcommand{\sym}{\mathrm{sym}\,}
\newcommand{\skewsym}{\mathrm{skew}\,}

\newcommand{\Lip}{M}
\newcommand{\admissible}{\mathcal A}

\newcommand{\energyfunction}{\mathcal E}

\newcommand{\lina}{\mathbb A^h}

\newcommand{\axl}{\mathrm{axl}}

\newcommand{\twist}{\mathrm{t}}
\newcommand{\pro}{\mathfrak{p}}

\newcommand{\I}{\mathrm {I}}
\newcommand{\II}{\mathrm {II}}
\newcommand{\III}{\mathrm {III}}

\newcommand{\eps}{\varepsilon}
\newcommand{\weakly}{\rightharpoonup}

\newcommand{\down}{\searrow}

\let\union\cup

\DeclarePairedDelimiter\abs{\lvert}{\rvert}
\DeclarePairedDelimiter\norm{\lVert}{\rVert}
\DeclarePairedDelimiter\set{\{}{\}}
\DeclarePairedDelimiter\paren{(}{)}

\DeclarePairedDelimiter\brackets{[}{]}

\DeclarePairedDelimiterXPP\Lpnorm[2]{}\lVert\rVert{_{L^{#1}}}{#2}
\DeclarePairedDelimiterXPP\Lebesgue[1]{\mathcal L^1}(){}{#1}
\DeclarePairedDelimiterXPP\Lebesguee[1]{\mathcal L^2}(){}{#1}
\DeclarePairedDelimiterXPP\Lebesgueee[1]{\mathcal L^3}(){}{#1}
\DeclarePairedDelimiterXPP\Lebesguen[1]{\mathcal L^n}(){}{#1}
\DeclarePairedDelimiterXPP\Hausdorff[1]{\mathcal H^1}(){}{#1}
\providecommand{\charact}[1]{\chi_{_{#1} }}

\newcommand\restrict[2]{{% we make the whole thing an ordinary symbol
  \left.\kern-\nulldelimiterspace % automatically resize the bar with \right
  #1 % the function
  \vphantom{\big|} % pretend it's a little taller at normal size
  \right|_{#2} % this is the delimiter
  }}

\makeatletter
\renewcommand*\env@matrix[1][*\c@MaxMatrixCols c]{%
  \hskip -\arraycolsep
  \let\@ifnextchar\new@ifnextchar
  \array{#1}}
\makeatother

\begin{document}
%{\bfseries \tableofcontents}
\author{Matthäus Pawelczyk \\ FB Mathematik, TU Dresden \\01062 Dresden (Germany) }
\date{}
\title{Convergence of equilibria for bending-torsion models of rods with inhomogeneities}
{\let\newpage\relax\maketitle}
\begin{abstract}
We prove that, in the limit of vanishing thickness, equilibrium configurations of inhomogeneous,
three-dimensional non-linearly elastic rods converge to equilibrium configurations of the variational limit theory. 
More precisely, we show that, as $h \down 0$, stationary points of the energy $E^h$, for a rod $\Omega_h \subset \R^3$ with cross-sectional diameter $h$, subconverge to
stationary points of the $\Gamma$-limit of $E^h$, provided that the bending energy of the sequence scales appropriately. 
This generalizes earlier results for homogeneous materials to the case of materials with (not necessarily periodic) inhomogeneities.
\end{abstract}

\section{Introduction}

The derivation of asymptotic models for two or three-dimensional elastic objects by lower-dimensional models has a long history, going back as far as to Bernoulli~\cite{bern1692} and Euler~\cite{euler1744}. Both considered thin rods, but starting from a two-dimensional model instead of the three-dimensional one, as we study here. Since then a multitude of such models has been proposed, some incompatible with each other. They usually depend on strong a priori assumptions. An in-depth study of the early history can be found in~\cite{levien08}. 

\vspace{\baselineskip}

We start with the nonlinear three-dimensional model: Let $\Omega_h \subset \R^3$ be the reference configuration of a thin elastic body, with `thickness' $h > 0$. The stored elastic energy of a deformation $y: \Omega_h \to \R^3$ is then given by
\[
  E^h(y) := \int_{\Omega_h} W( \nabla y(x)) \td x,
\]
where $W$ is the \textit{elastic energy density}; typical assumptions on $W$ are similar to those provided in~\ref{(M1)}--\ref{(M4)}. One is interested in the limiting behaviour of $E^h$ as $h\down0$.
One of the first results in terms of $\Gamma$-convergence were for $\Omega_h := \omega \times (-h,h)$ with $\omega \subset \R^2$.
Roughly speaking $\Gamma$-convergence is equivalent to the convergence of global minimizers $y^h$ of $E^h$, possibly perturbed by some forcing term, to global minimizers of some limiting energy.
For example in~\cite{LDR96} \textit{the theory for membranes}, i.e., the limit for $h^{-1} E^h$ was obtained, in~\cite{FJM02} the \textit{bending theory for plates}, i.e., for $h^{-3} E^h$.
The latter result contains, as a particular case, the model proposed by Bernoulli and Euler. Further scalings $h^{-\alpha} E^h$ were later studied in~\cite{FJM06}. In this present paper we study rods with small cross-sectional diameter. So in our case the reference configuration is $\Omega_h := (0,L) \times h\omega$ for some $L > 0$ and $\omega \subset \R^2$. The \textit{bending-torsion theory for rods}, i.e., the $\Gamma$-limit for $h^{-3} E^h$ was obtained by~\cite{MM02}. 
Under the additional assumption of a linear stress growth, the result was strengthened in~\cite{MM08} by proving that also stationary points $y^h$ of $E^h$ subconverge to stationary points of the $\Gamma$-limit.

All the previous mentioned results were obtained in the case of a single, homogeneous material. In~\cite{N10} the first $\Gamma$-convergence result for a rod in this regime, i.e. $h^{-3}E^h$, with inhomogeneities was proved. This was done under the assumption that the inhomogeneity was periodic, rapidly oscillating and only depending on the `in-plane' variable $x_1 \in (0,L)$. All these additional assumptions can be dropped, as was shown in~\cite{MV16}. In the present paper we extend the result of~\cite{MV16} by showing that also stationary points subconverge to stationary points of the $\Gamma$-limit.

In~\cite{BVP17} the more linear case of $h^{-5} E^h$, called the \textit{von K\'arm\'an model}, was studied, and $\Gamma$-convergence and convergence of stationary points was proved. This result, and the one presented here, heavily depend on methods developed in~\cite{MV16,VvK17}.

\vspace{\baselineskip}

Now we turn to the precise mathematical description. Let $L > 0$ and let $\omega \subset \R^2$ be open, bounded, connected. The (scaled) energy of a non-homogeneous rod with length $L$ and cross-section $h\omega$ and external forces $g \in L^2((0,L), \R^3)$, deformed by $y: [0,L] \times h \omega \to \R^3$, is given by
\begin{equation*}\begin{aligned}
  \widetilde \energyfunction^h( y ) = \frac 1 {h^4} \int_{ [0,L] \times h\omega} W^h\paren[\big]{ (x_1, h^{-1} x'), \nabla y(x)} \td x - \frac 1 {h^2}\int_{[0,L] \times h\omega} g(x_1) \cdot y(x) \td x.
\end{aligned}\end{equation*}
The hypotheses on the elastic energy density $W^h : [0,L] \times \omega\times \R^{3\times 3}\to[0,\infty)$ are listed in Section~\ref{subsec:assumptions}. After performing the usual change of variables $(x_1, x_2, x_3)\mapsto (x_1, hx_2, hx_3)$ the rod $\Omega_h$ scales to $\Omega := \Omega_1$ and we obtain 
\begin{equation}\begin{aligned}\label{eq:energyhrescaled}
  \energyfunction^h( y ) = \frac 1 {h^2} \int_{ [0,L] \times \omega } W^h(x, \nabla_h y(x)) \td x - \int_{[0,L]\times \omega} g(x_1) \cdot y(x) \td x,
\end{aligned}\end{equation}
where $\nabla_h = (\partial_1, \frac 1h \partial_2, \frac 1 h \partial_3)$.
As already mentioned, in~\cite{MV16} the $\Gamma$-convergence of $\energyfunction^h$ along a subsequence to a limiting functional $\energyfunction^0$ was proved. This limit is given by
\begin{equation}\begin{aligned}\label{eq:energydefinition}
  \energyfunction^0( y, d_2, d_3 ) & := \begin{cases}
  \int_0^L Q_1^0( x_1, R^T(x_1)R'(x_1)) -g(x_1) \cdot y(x_1) \td x_1 & \text{ if } (y,d_2,d_3) \in \mathcal A, \\
\infty & \text{ else,}
                                      \end{cases}
\end{aligned}\end{equation}
where $Q_1^0$ is a quadratic form in the second argument, which will be introduced in Proposition~\ref{prop:density}; the class of limiting deformations~$\admissible$ is given by
\begin{equation}\begin{aligned}\label{eq:admissibleset}
  \admissible := &\set {(y, d_2, d_3) \in W^{2,2}((0,L),\R^3)\times W^{1,2}{((0,L),\R^3)} \times W^{1,2}((0,L), \R^3):\\
          &\hspace{6cm} \quad (y', d_2, d_3) \in W^{1,2}( (0,L),\SO(3)) },
\end{aligned}\end{equation}
equipped with the strong $W^{2,2}\times W^{1,2} \times W^{1,2}$-topology, 
and $R = (y', d_2, d_3)$ is the rotation associated with $( y, d_2, d_3 )$.

Formally the first variation of the energy functional $\energyfunction^h$ in direction of some test function $\psi: [0,L]\times \omega \to \R^3$ is given by
\begin{equation}\begin{aligned}\label{eq:variationenergyh}
  D\energyfunction^h( y )[\psi] :=\frac 1 {h^2} \int_{ [0,L] \times \omega } DW^h(x, \nabla_h y(x)):\nabla_h \psi \td x - \int_{[0,L]\times \omega} g(x_1) \cdot \psi(x) \td x.
\end{aligned}\end{equation}
For the first integral to be well-defined, however, we need to impose linear stress growth, i.e., for any $F \in \R^{3\times3}$ we require the inequality $\abs{DW^h(\cdot, F)}\leq L(\abs F + 1)$ to hold. Deformations $y$ satisfying $D\energyfunction^h(y)[\psi] = 0$ for all test functions $\psi$ are said to be stationary.  
 If we impose the boundary condition $y(0,x') = (0, hx')$ then the natural class of test functions $\psi$ in~\eqref{eq:variationenergyh} are $C^{\infty}(\overline {\Omega}, \R^3)$ maps, which vanish at $\set 0 \times \omega$; we denote this class by $C^\infty_{\bdy}(\overline {\Omega}, \R^3)$.
Another notion of stationary points exists, introduced by J.~Ball in~\cite{JB84}, which does not need linear stress growth, and is furthermore compatible with physical growth, i.e., $W(F) \to \infty$ if $\det F \down 0$ and $W(F) = \infty$ if $\det F \leq 0$. 
In~\cite{DM12} the convergence of such stationary points for the von K\'arm\'an rod (for homogeneous materials) was shown. Due to the highly inhomogeneous material we will need to stay in the first setting. Regardless of the notion of stationarity, and even for homogeneous materials, the existence of stationary points is a subtle issue, see~\cite[section 2.2, section 2.7]{JB02}.

\vspace{\baselineskip}

For $\alpha,\beta, \Lip$ positive constants with $\alpha\leq \beta$ we denote by $\mathcal W(\alpha,\beta, \Lip)$ the set of admissible density functions $W^h$; the precise definition of the class $\mathcal W(\alpha,\beta, \Lip)$ is given by~\ref{enum:S1}--\ref{enum:S3} below.
We can now state the main result of this paper:
\begin{thm}

\label{thm:main}
Let $(W^h) \subset \mathcal W(\alpha,\beta, \Lip)$, $g \in L^2((0,L), \R^3)$ and $(y^h) \subset W^{1,2}(\Omega, \mr^3)$, such that $y^h(0,x_2,x_3) = (0, hx_2, hx_3)$ for any $h >0$, and furthermore
\begin{equation}\begin{aligned}\label{eq:energyboundinthm}
\limsup_{h\down 0} \frac 1 {h^2} \int_\Omega W^h(x, \nabla_h y^h(x)) \,\td x < \infty.
\end{aligned}\end{equation}
Assume in addition, that each $y^h$ is a stationary point of $\energyfunction^h$, given in~\eqref{eq:energyhrescaled}, subject to natural boundary conditions, i.e., $D\energyfunction^h[y^h][\psi]=0$ for all $\psi \in C^\infty_{\bdy}(\overline \Omega,\R^3)$. 
 Then there exists $(\overline y, \overline d_2, \overline d_3 )\in \admissible$, such that, up to a subsequence, $y^h \to \overline y$ strongly in $W^{1,2}(\Omega, \R^3)$ as $h\down 0$, and 
\[
  \nabla_h y^h \to (\overline y', \overline d_2, \overline d_3) \quad \text{ strongly in } L^2(\Omega, \R^{3\times3}).
\]
Furthermore $\overline y(0) = 0$, $\overline d_k(0) = e_k$ for $k=2,3$, and $(\overline y, \overline d_2, \overline d_3) $ is a stationary point of $\energyfunction^0$, where $\energyfunction^0$ is given in~\eqref{eq:energydefinition}.
\end{thm}

\begin{remark}
It is easily seen that there are sequences $(y^h)$ satisfying the boundary conditions $y^h(0, x_2, x_3) = hx_2 e_2 + hx_3 e_3$ such that~\eqref{eq:energyboundinthm} holds.
Thus an application of Poincar\'e's inequality shows that~\eqref{eq:energyboundinthm} holds automatically for a minimizing sequence $(y^h)$.
\end{remark}
\begin{remark}
The theorem also holds true for the more general forces $\widetilde g \in L^2( \Omega,\R^3)$. For this the forces in the limiting energy must be replaced by the mean of $\widetilde g$ on $\omega$, i.e., by $\int_\omega g(\cdot, x') \td x'$.
The more general statement can be proved identically, up to a few additional error terms, but which converge trivially to zero for $h\down 0$.
\end{remark}

The proof of Theorem~\ref{thm:main} is split into two main parts. 
For the first one we follow closely the paper~\cite{MM08}, where the corresponding result for the homogeneous rod was proved. 
Their methods for studying the stress can also be applied, with minor modifications, in the more general case considered here. 
Furthermore we use additional cancellation effects, which simplifies parts of their proof. 
To conclude their proof they exploit an explicit, linear relationship between the limiting stress and strain, which allows to easily identify the limit equations. 
In the inhomogeneous case addressed here, such a relationship is less clear and the identification of the limit equation is more involved. 
Thus for the second part we apply results and methods developed in~\cite{BVP17} to identify the limit equation and conclude the proof.

\section{Preliminaries}
\subsection{Notation}
Let $x = (x_1, x') \in \R^3$, and let $\pro (x) = (0, x') \in \R^3$ be the projection of $x$ onto $\set 0 \times \omega$. Let ${(e_i)}_{i=1}^3$ be the standard basis of $\R^3$. By $(\cdot)$ we denote the inner product on $\R^3$ and by $(:)$ the inner product on $\R^{3\times 3}$, i.e., $A : B =\tr (A^T B)$ for any $A,B \in \R^{3\times 3}$, with $\tr$ being the trace.
The twist function $\twist: L^1(\Omega)^2 \to L^1(0,L)$ is given by $\twist(\phi, \psi)(x_1) = \int_\omega x_3 \phi(x_1,x') - x_2\psi(x_1,x')\td x'$.
We denote by $\iota: \mr^3 \to \mr^{3\times3}$ the natural inclusion $\iota(v) = v \otimes e_1$, by $\axl: \mr^{3\times3}_{\skewsym} \to \mr^3$ the axial vector 
$\axl(A) = (-A_{23}, A_{13}, -A_{12})$ and by $\id_{3\times 3}$ we denote the $3\times 3$ identity matrix. By $()'$ we denote the derivative with respect to $x_1$,  by $\nabla = (\partial_1, \partial_2, \partial_3)$ the gradient with respect to $x$ and for every $h>0$ we define the scaled gradient as $\nabla_h = (\partial_1, \frac 1 h\partial_2, \frac 1 h \partial_3 )$. 
For the test functions we define $C^\infty_{\bdy}([0,L]) = \set { f\in C^\infty([0,L]) | f(0) = 0 }$ and $C^\infty_{\bdy}(\overline \Omega) := C^\infty_{\bdy}([0,L], C^\infty(\overline \omega))$, and $W^{1,2}_{\bdy}([0,L]) = \set { f\in W^{1,2}([0,L]) | f(0) = 0 }$ and $W^{1,2}_{\bdy}(\Omega) = W^{1,2}_{\bdy}([0,L], W^{1,2}(\omega))$.

\subsection{The nonlinear bending-torsion theory for beams}\label{subsec:assumptions}
Let $L > 0$ and let $\omega \subset \R^2$ be an open, bounded, connected Lipschitz-domain with $\Lebesguee{\omega} = 1$ and which is centered, i.e.,
\begin{equation}\begin{aligned}\label{eq:omegacentered}
  \int_\omega x_2 x_3 \,\td x_2 \td x_3
 = 
  \int_\omega x_2 \, \td x_2 \td x_3 = 
  \int_\omega x_3 \,\td x_2 \td x_3
 = 0.
\end{aligned}\end{equation}
The reference domain $\Omega$ is given by $\Omega = (0,L) \times \omega$. The assumption on the elastic energy density $W$ are as follows:

Let $\alpha, \beta, \Lip$ be positive constants with $\alpha \leq \beta$. The class $\mathcal W(\alpha,\beta, \Lip)$ 
contains all differentiable functions $W : \R^{3\times 3} \to [0, \infty)$ that satisfy:
\begin{enumerate}[label= (M\arabic*)]
\item\label{(M1)} Frame indifference: $W(RF) = W(F)$ for all $F \in \mr^{3\times 3}$ and $R \in \SO(3)$.
\item\label{(M2)} Non-degeneracy and continuity:
\begin{align*}	
   \alpha \dist^2(F, \SO(3)) \leq W(F) \leq \beta  \dist^2(F, \SO(3))& \quad \text{ for all } F \in \R^{3\times3}.
\end{align*}
Note that this implies the minimality at the identity, i.e., $W(\id_{3\times3}) = 0$.
\item\label{(M4)} Linear stress growth: For the derivative $DW$ of $W$ we have: $\abs { DW(F) } \leq \Lip(\abs F + 1)$ for all $F \in \R^{3\times 3}$.
\end{enumerate}
\begin{remark}
The condition~\ref{(M4)} is needed here for the first term in the first variation of $\energyfunction^h$, given in~\eqref{eq:variationenergyh}, to be well-defined, and thus the condition appears in similar form in~\cite{MM08,MMS06}. It is however not needed for results concerning $\Gamma$-convergence, e.g.,~\cite{MV16}. There also the upper bound~\ref{(M2)} is only needed locally, i.e.,
\begin{align*}	
\exists \rho, \beta' > 0 : \; W(F) \leq \beta'  \dist^2(F, \SO(3))& \; \text{ for all } F \in \R^{3\times3} \text{ with } \dist(F, \SO(3)) \leq \rho.
\end{align*}
It is however easily seen that this local upper bound together with linear stress growth implies the global estimate~\ref{(M2)} for some $\beta > 0$.
\end{remark}

Let now $\alpha,\beta, \Lip$ be as above. A family of energy densities ${(W^h)}_{h >0}$, $W^h: \Omega \times \mr^{3\times3} \to [0,\infty)$ describes 
an \textit{admissible composite material of class $\mathcal W(\alpha,\beta,\Lip)$} if for every $h > 0$ it holds:
\begin{enumerate}[label= (S\arabic*)]
\item\label{enum:S1} $W^h$ is a Borel function on $\Omega \times \R^{3\times3}$.
\item  $W^h(x,\cdot) \in \mathcal W(\alpha,\beta, \Lip)$ for almost every $x \in \Omega$.
\item\label{enum:S3} There exist a monotone function $r: [0,\infty] \to [0, \infty]$ and quadratic forms $Q^h: \Omega \times \R^{3\times3} \to [0,\infty)$ such that $r(\eps) \down 0$ if $\eps \down 0$ and
\[
   \esssup_{x \in \Omega} \abs { W^h(x, \id_{3\times 3} + G) - Q^h(x, G)} \leq r( \abs G) \abs G^2 \quad \text{ for all } G \in \mr^{3\times 3}.
\]
\end{enumerate}
Let $(Q^h)$ be the family of corresponding quadratic forms associated with a family $(W^h) \subset \mathcal W(\alpha, \beta, L)$, then it is easy to see that
for every $h > 0$ we have:

 $Q^h$ is a Carath\'eodory function, which for almost every $x \in \Omega$ satisfies
\begin{equation}\begin{aligned}\label{eq:Qproperties}
&\alpha\abs{ \sym F}^2 \leq Q^h(x, F) = Q^h(x, \sym F) \leq \beta \abs {\sym F}^2            &&\quad \text{ for all } F \in \mr^{3\times3}, \\
    &\abs{ Q^h(x, F_1) - Q^h(x, F_2)} \leq \beta \abs {\sym F_1 - \sym F_2}&& \\
  &\hspace{5cm}\cdot\abs{\sym F_1 + \sym F_2}&&\quad \text{ for all } F_1,F_2 \in \mr^{3\times3}.
    \end{aligned}\end{equation}
Let $\lina$ denote the linear, symmetric, positive semidefinite operator associated with the quadratic forms $Q^h$, i.e., $Q^h(F) = \frac 1 2 \lina F: F$ for all $F \in \R^{3\times3}$.

In~\cite[proposition 4.1]{MM08} the following compactness result was shown:
\begin{prop}\label{thm:compactness}
Let $(u^h) \subset W^{1,2}(\Omega,\mr^3)$ be a sequence satisfying
\begin{equation}\begin{aligned}\label{eq:energybound}
\limsup_{h\down 0} \frac 1 {h^2} \int_\Omega \dist^2\paren[\big]{\nabla_h u^h, \SO(3)}\td x < \infty.
\end{aligned}\end{equation}
Then there exists a constant $C>0$, depending only on the domain $\Omega$, and a sequence $(R^h) \subset C^\infty([0,L], \SO(3))$, such that
\begin{align}	
  \label{eq:estimatedifferenceutorot}\norm { \nabla_h u^h - R^h}_{L^2(\Omega)} \leq Ch \\
  \label{eq:estimatederivativeR}\norm{ (R^h)' }_{L^2((0,L))} + h\norm{ (R^h)'' }_{L^2((0,L))} \leq C
\end{align}
for every $h  > 0$. If, in addition, $u^h(0,x_2,x_3) = (0,hx_2, hx_3)$, then
\begin{equation}\begin{aligned}\label{eq:Rbdry}
  \abs{ R^h(0) - \id } \leq C \sqrt h.
\end{aligned}\end{equation}
\end{prop}

The following observations are standard, and follow the approach taken in~\cite{MM08}. 
Let $(y^h)$ be the sequence of deformations satisfying the assumptions of Theorem~\ref{thm:main}. The non-degeneracy assumption~\ref{(M2)} implies that $(y^h)$ satisfies~\eqref{eq:energybound}. Thus by the previous proposition there exists a sequence $(R^h)$ satisfying~\eqref{eq:estimatedifferenceutorot} and~\eqref{eq:estimatederivativeR}. By using the frame-indifference of $W^h$ we have
\begin{align*}	
  W^h(\cdot, \nabla_h y^h) = W^h(\cdot, (R^h)^T \nabla_h y^h) &= W^h\paren*{\cdot, \id_{3\times 3} +h \frac{ (R^h)^T \nabla_h y^h -\id_{3\times 3} }h}\\
  &= W^h\paren*{\cdot, \id_{3\times 3} +h G^h},
\end{align*}
where we introduced
\begin{equation}\begin{aligned}\label{eq:Ghdefinitionu}
  G^h = \frac { (R^h)^T \nabla_h y^h - \id_{3\times3}} h.
\end{aligned}\end{equation}
The estimate~\eqref{eq:estimatedifferenceutorot} implies that $(G^h)$ is uniformly bounded in $L^2$.
We define $z^h$ implicitly by introducing the ansatz
\begin{equation}\begin{aligned}\label{eq:zhintro}
  y^h(x) = \int_0^{x_1} R^h(s) e_1 \td s+ hx_2 R^h(x_1) e_2 + h x_3 R^h(x_1) e_3 + hz^h(x).
\end{aligned}\end{equation}

Inserting this ansatz into~\eqref{eq:Ghdefinitionu} we can calculate  that
\begin{equation}\begin{aligned}\label{eq:GsplittingMM}
  G^h = \frac { (R^h)^T \nabla_h y^h - \id_{3\times3}} h 
      =\iota\paren[\big]{  A^h \pro} + (R^h)^T\nabla_h z^h,
\end{aligned}\end{equation}
where we introduced $A^h := (R^h)^T (R^h)'$. Clearly $(A^h)$ is uniformly bounded in $L^2$, and since $(G^h)$ is uniformly bounded in $L^2$ as well, the sequence $(\nabla_h z^h)$ is uniformly bounded in $L^2$. Furthermore on $\set 0 \times \omega$ we have the boundary conditions $y^h(x) = hx_2 e_2 + hx_3 e_3$, and thus also we can assume~\eqref{eq:Rbdry} holds. With this we obtain $\abs { z^h } \leq C\sqrt h$ on $\set 0 \times \omega$. By applying Poincar\'e's inequality we can now find a uniform bound on the $L^2$-norm of $z^h$, and thus on the $W^{1,2}$-norm of $z^h$.

Thus, after extracting a subsequence, which we will not relabel, we have in $L^2$ the weak convergences 
\[
  G^h \weakly G, \quad, \quad A^h \weakly A \quad\text{ and }\quad (R^h)^T\nabla_h z^h \weakly R (\partial_1 z | q_2 | q_3)
\]
for some $G \in L^2(\Omega, \R^{3\times3})$, $z \in W^{1,2}(\Omega,\R^3)$, $A \in L^2((0,L), \R^{3\times3}_{\skewsym})$ and $q_2,q_3 \in L^2(\Omega, \R^{3})$. Notice that the uniform $L^2$ bound on $( \frac 1 h\partial_2 z^h, \frac 1 h \partial_3 z^h )$  implies that $z$ does not depend on $x_2,x_3$. Thus going to the limit in~\eqref{eq:GsplittingMM} we obtain
\begin{equation*}\begin{aligned}
G(x) = \iota\paren[\big]{p(x_1) + A(x_1) \pro(x)} + R(x_1)^T (0 | q_2(x) | q_3(x) ),
  \end{aligned}\end{equation*}
where for brevity we set $p := R^T \partial_1 z \in L^2((0,L), \R^3)$. Next we focus on $\sym G^h$. In~\cite[proof of theorem 2.15]{MV16} it is shown that there exist sequences $v^h \subset W^{1,2}(\Omega, \R^3)$, $(\Psi^h) \subset W^{1,2}((0,L),\R^{3\times3}_{\skewsym})$ and $o^h \subset L^2(\Omega, \R^{3\times3})$ such that 
\begin{equation*}\begin{aligned}
  \sym G^h = \sym \iota\paren[\Big] { A\pro + p_1 e_1} + \sym \iota \paren[\Big]{(\Psi^h)'\pro} + \sym \nabla_h v^h + o^h,
\end{aligned}\end{equation*}
and such that $( \nabla_h v^h )$ is uniformly bounded in $L^2$, and $ \Psi^h \weakly 0$ in $W^{1,2}((0,L), \mr^{3\times3})$,  $v^h \weakly 0$ in $W^{1,2}(\Omega, \R^3)$ and $o^h \to 0$ strongly in $L^2(\Omega,\R^{3\times3})$.

We define the {\itshape fixed part} $m_d$ by
\begin{equation}\begin{aligned}\label{eq:definitionmd}
m_d := A\pro + p_1 e_1
\end{aligned}\end{equation}
and the {\itshape corrector sequence} $(\psi^h)$ by
 \begin{equation}\begin{aligned}\label{eq:definitioncorrector}
 \psi^h(x) = \Psi^h \pro - \frac 1 h \paren*{ \widehat \Psi^h_{12}e_2 +   \widehat \Psi^h_{13}e_3}   + v^h,
 \end{aligned}\end{equation}
where $\widehat \Psi^h(x_1) = \int_0^{x_1} \Psi^h(s) \td s$. Direct calculation yields
\[
  \nabla_h \psi^h = \begin{pmatrix} (\Psi^h)' \pro - \frac 1h \Psi^h_{12} e_2 - \frac 1 h \Psi^h_{13}e_3, \quad \frac 1 h \Psi^he_2, \quad \frac 1 h \Psi^h e_3\end{pmatrix} + \nabla_h v^h,
\]
as well as
\begin{equation}\begin{aligned}\label{eq:sympsidecomp}
\sym \nabla_h \psi^h = \sym \iota \paren[\Big]{(\Psi^h)'\pro} + \sym \nabla_h v^h.
\end{aligned}\end{equation}
Thus we have
\begin{equation}\begin{aligned}\label{eq:symGhdecomp}
  \sym G^h = \sym \iota\paren[\big] { m_d} + \sym  \nabla_h \psi^h + o^h,
\end{aligned}\end{equation}
with easily verifiable strong convergences
\[
  (\psi^h_1, h \psi^h_2, h\psi^h_3) \to 0 \; \text{ in } L^2( \Omega, \mr^3) \quad \text{ and } \quad \twist(\psi^h_2, \psi^h_3) \to 0  \text{ in } L^2((0,L)).
\]

\subsection{The $\Gamma$-limit}
We will briefly introduce the variational approach developed in~\cite{MV16}, with which the $\Gamma$-convergence for the inhomogeneous rod was proved. A similar variational approach for thin elastica was used earlier in~\cite{BFF00} for the membrane model. The approach was also already adapted and used in~\cite{BVP17} to show the convergence of stationary points for the inhomogeneous von K\'arm\'an rod.

\vspace{\baselineskip}

By applying the frame indifference~\ref{(M1)} and Taylor expansion~\ref{enum:S3} we obtain
\begin{equation}\begin{aligned}\label{eq:Wtaylorexp}
\frac 1 {h^2}W^h(\cdot, \nabla_h y^h) & =\frac 1 {h^2} W^h\paren*{\cdot, \id_{3\times 3} +h G^h}\\
                                      & \approx \frac 1 {h^2}Q^h(\cdot, hG^h) = Q^h(\cdot, G^h)= Q^h(\cdot, \sym G^h).
\end{aligned}\end{equation}
This motivates, together with the decomposition~\eqref{eq:symGhdecomp}, the definitions
\begin{align*}	
  \mathcal K^-_{(h)}(m,O):= \inf \set*{ \liminf_{h\down 0} \int_{O \times \omega} Q^h(x, \iota(m) + \nabla_h \psi^h) \td x},\\
  \mathcal K^+_{(h)}(m,O):= \inf \set*{ \limsup_{h\down 0} \int_{O \times \omega} Q^h(x, \iota(m) + \nabla_h \psi^h) \td x},
\end{align*}
where we take the infimum over all sequences $(\psi^h) \subset W^{1,2}(O\times \omega,\mr^3)$ such that
\[
  (\psi^h_1, h \psi^h_2, h\psi^h_3) \to 0 \; \text{ strongly in } L^2(O \times \omega, \mr^3) \quad \text{ and } \quad \twist(\psi^h_2, \psi^h_3) \to 0 \quad \text{ in } L^2(O).
\]

It is proved in~\cite[lemma~2.6]{MV16} that there exists a subsequence, still denoted by $(h)$, such that:
\begin{equation}\begin{aligned}\label{eq:defabstrlimit}
\forall m\in L^2(\Omega,\R^3), \forall O \subset [0,L] \text{ open } : \mathcal K_{(h)}(m,O) := \mathcal K^-_{(h)}(m,O) = \mathcal K^+_{(h)}(m,O).
\end{aligned}\end{equation}
This can be done by extracting a diagonal sequence such that $\mathcal K^-_{(h)}$ and $\mathcal K^+_{(h)}$ agree on a dense, countable subset of $L^2$ and of open subsets of $(0,L)$. Utilizing the continuity of the maps $L^2(\Omega, \mr^3) \to \mr$, $m \mapsto \mathcal K^-_{(h)}(m,O),m\mapsto \mathcal K^+_{(h)}(m,O)$ for any open set $O\subset(0,L)$, proved in~\cite[lemma~2.5]{MV16}, it is then easy to see that~\eqref{eq:defabstrlimit} holds.

We now introduce the relaxation sequence and state its most important properties, which were proved in~\cite{BVP17} and~\cite{MV16}:
\begin{lemma}\label{lem:orthabstract}
Let $(h) \subset (0, \infty)$ with $h \down 0$ be a sequence such that~\eqref{eq:defabstrlimit} holds true.
Then there exists a subsequence (not relabeled) such that for every $m\in L^2(\Omega, \mr^2)$ there exists
$(\psi^h_m)\subset W^{1,2}(\Omega,\mr^2)$, with the property that for every open set $O \subset (0,L)$ we have
\begin{equation}\begin{aligned}\label{eq:propofpsim}
\mathcal K_{(h)}(m,O):= \lim_{h\down 0} \int_{O \times \omega} Q^h(x, \iota(m) + \nabla_h \psi^h_m) \td x,
  \end{aligned}\end{equation}
and $(\psi^h_m)$ satisfies the following properties:
\begin{enumerate}[label=\emph{(\alph*)}]
\item $(\psi^h_m \cdot e_1, h\psi^h_m \cdot e_2, h\psi^h_m \cdot e_3) \to 0$ and $\twist(\psi^h_m \cdot e_2, \psi^h_m \cdot e_3) \to 0$ strongly in $L^2$.
\item The sequence $(\abs{ \sym \nabla \psi^h_m }^2)$ is equi-integrable, and there exist sequences 
$B^h_m \subset W^{1,2}( (0,L),\R^{3\times3}_{\skewsym})$ and $(\vartheta^h_m) \subset W^{1,2}(\Omega,\R^3)$ with $B^h_m \to 0$, $\vartheta^h_m \to 0$ strongly in their respective $L^2$-norm,
and
\[
 \sym \nabla_h \psi^h_m = \sym\, \iota \paren[\big]{(B^h_m)'\pro} + \sym \nabla_h \vartheta^h_m.
\]
Moreover, for a subsequence $( \abs { (B^h_m)'}^2)$ and $(\abs{ \nabla_h \vartheta^h_m}^2 )$ are equi-integrable, and the following inequality holds for some $ C>0$ independent of $O \subset (0,L)$:
\[
  \limsup_{h \down 0} \paren* { \norm{ B_m^h  }_{W^{1,2}(O)} + \norm{ \nabla_h \vartheta^h_m}_{L^2(O\times \omega)} } \leq C ( \beta \norm { m }_{L^2(O \times \omega)}^2 + 1).
\]
\item If $(\widehat \psi^h) \subset W^{1,2}(\Omega, \mr^3)$ is any other sequence that satisfies (a) and $(\sym \nabla_h \widehat \psi^h)$ is bounded in $L^2(\Omega,\mr^{3\times3})$,
then
\[
  \lim_{h \down 0} \int_\Omega \lina (\iota(m) + \nabla_h \psi^h_m) : \sym \nabla_h \widehat \psi^h \td x = 0.
\]

\item If $(\widehat \psi^h) \subset W^{1,2}(\Omega,\mr^3)$ is any sequence that satisfies~\eqref{eq:propofpsim} and (a), then
\[
  \norm { \sym \nabla_h \psi^h_m - \sym \nabla_h \widehat \psi^h }_{L^2(\Omega)} \to 0,
\]
and $(\abs { \sym \nabla_h \widehat \psi^h}^2 )$ is equi-integrable.

\item\label{enum:EL} The map $\mathcal K_{(h)}( \cdot, (0,L)): L^2(\Omega, \mr^3) \to \R$ is continuously Fr\'echet-differentiable, and for every $m,n \in L^2(\Omega, \mr^3)$ we have
\begin{equation}\begin{aligned}\label{eq:mderivativeK}
\frac { \partial \mathcal K_{(h)}(m, (0,L))}{\partial m}[n] = \lim_{h\down 0}\int_\Omega \lina(\iota(m) + \sym \nabla_h \psi^h_m): \iota(n)\td x.
                                                                                                                                    \end{aligned}\end{equation}
\end{enumerate}
\end{lemma}
The sequence $(\psi^h_m)$ is called the {\itshape relaxation sequence} for $m$. For our purposes $m$ will always be of the form $m = B \pro + b e_1$ for some $B \in L^2((0,L), \mr^{3\times 3}_{\skewsym})$ and $b \in L^2((0,L))$. Thus we introduce the linear map
\begin{equation}\begin{aligned}\label{eq:mdefinition}
m:L^2((0,L), \mr^{3\times 3}_{\skewsym}) \times L^2((0,L),\R) \to L^2(\Omega, \R^3), \quad m(B,b) :=  B \pro + b e_1.
\end{aligned}\end{equation} 

By applying the chain rule together with~\ref{enum:EL} we thus can deduce the derivative of $\mathcal K_{(h)}(\cdot, (0,L)) \circ m$. 
Indeed, for every $B,M \in W^{1,2}((0,L),\R^{3\times3}_{\skewsym})$, $b,\mu\in L^2((0,L))$  we have 
\begin{equation}\begin{aligned}\label{eq:limitELnonopt}
\frac { \partial \mathcal K_{(h)}(m(B,b), (0,L))}{\partial B}[M]&= \lim_{h\down 0}\int_\Omega \lina(\iota(m(B,b)) + \sym \nabla_h \psi^h_{m(B,b)}): \iota(M\pro)\td x,\\
                                                                                                                                                       \frac { \partial \mathcal K_{(h)}(m(B,b), (0,L))}{\partial b}[\mu]&= \lim_{h\down 0}\int_\Omega \lina(\iota(m(B,b)) + \sym \nabla_h \psi^h_{m(B,b)}): \iota(\mu e_1)\td x.
\end{aligned}\end{equation}
To shorten notation we define $\mathcal K(m) := \mathcal K(m, (0,L))$ for every $m\in L^2$.

In~\cite[proposition 2.12]{MV16} also the following result regarding the existence of a density for $\mathcal K_{(h)}$ was proved:
\begin{prop}\label{prop:density}
Let $(h) \subset (0, \infty)$ with $h \down 0$ be a sequence such that~\eqref{eq:defabstrlimit} holds true for every $m \in L^2(\Omega,\R^3)$. Then a measurable function $Q^0: [0, L] \times \R^{3\times3}_{\skewsym} \times \R \to [0, \infty)$ exists, such that for every $O \subset [0,L]$ open, and every $(B,b) \in L^2((0,L), \R^{3\times3}_{\skewsym} \times \R)$ we have
\begin{equation*}\begin{aligned}
  \mathcal K_{(h)}(m(B,b),O) = \int_O Q^0(x_1, B(x_1), b(x_1)) \td x_1.
\end{aligned}\end{equation*}
Furthermore for almost every $x_1 \in [0,L]$ the map $Q^0(x_1, \cdot, \cdot)$ is a quadratic form,
and there exists $C = C(\omega) > 0$ independent of $x_1$, such that for every $\widehat B \in \R^{3\times 3}_{\skewsym},\widehat  b\in \R$ we have
\[
    C^{-1} ( \abs{ \widehat B}^2 + \abs{ \widehat  b }^2) \leq Q^0(x_1, \widehat B,\widehat b) \leq C \beta ( \abs{ \widehat B}^2 + \abs{ \widehat  b }^2).
\]
\end{prop}
From this we easily deduce that the map $\widehat b_{\min}: [0,L] \times \R^{3\times3}_{\skewsym} \to \R$, given by $\widehat b_{\min}(x_1, \widehat B) = \argmin_{b \in \R} Q^0(x_1,\widehat  B, b)$ is well-defined, linear in $\widehat B$, and there exists a constant $C' = C'(\alpha,\beta, \omega) > 0$ such that for almost every $x_1$ and all $\widehat B \in \R^{3\times 3}_{\skewsym}$ we have
\[
  \abs { \widehat b_{\min}(x_1, \widehat B) } \leq C'\abs{\widehat  B}.
\]
Finally we can define the density for the limiting bending energy. Let map $Q^0_1: [0,L] \times \R^{3\times3}_{\skewsym} \to \R$ be given by $Q_1^0(x_1, \widehat B) := Q^0(x_1, \widehat B,\widehat b_{\min}(x_1, \widehat B))$. It is easily seen that:

For almost every $x_1 \in [0,L]$ the map $Q_1^0(x_1, \cdot)$ is a quadratic form, and there exists $C'' = C''(\alpha,\beta,\omega) > 0$ such that for all $\widehat B \in \R^{3\times 3}_{\skewsym}$ it holds
\[
  (C'')^{-1} \abs {\widehat B}^2 \leq Q_1^0(x_1, \widehat B) \leq C'' \abs { \widehat B }^2.
\]
We now define the limiting bending energy $\mathcal K^0_{(h)}: L^2((0,L),\R^{3\times3}_{\skewsym}) \to \R$ simply by integrating over the density $Q^0_1$, i.e., 
\begin{equation*}\begin{aligned}
     \mathcal K^0_{(h)}(B) &:= \int_0^L Q^0_1\paren[\Big]{x_1, B(x_1)}\td x_1 \\
                                    &\phantom{:}= \int_0^L Q^0\paren[\Big]{x_1, B(x_1),\widehat  b_{\min}\paren[\Big]{x_1, B(x_1)}  }\td x_1.
\end{aligned}\end{equation*}
From the linearity of $\widehat B\mapsto\widehat  b_{\min}(\cdot,\widehat  B)$ and the Fr\'echet-differentiability of $\mathcal K_{(h)}$ we deduce that also $\mathcal K^0_{(h)}$ is Fr\'echet-differentiable. 
For fixed $\widehat B \in \R^{3\times 3}_{\skewsym}$ and almost every $x_1$ the function $Q^0(x_1, \widehat B, \cdot)$ has quadratic growth, thus $\widehat b_{\min}(x_1, \widehat B)$ is the unique stationary point of $Q^0(x_1, \widehat B, \cdot)$, i.e.,
\begin{equation*}\begin{aligned}
(\partial_b Q^0)(x_1, \widehat B, b) = 0 \iff b = \widehat b_{\min}(x_1, \widehat B).
\end{aligned}\end{equation*}

Furthermore the mapping $ b_{\min}:  L^2((0,L),\R^{3\times3}_{\skewsym}) \to L^2((0,L))$, $ b_{\min}( B) =\widehat b_{\min}(\cdot, B)$ is linear and well-defined. Thus for any $B \in L^2((0,L), \R^{3\times 3}_{\skewsym})$ and $b \in L^2((0,L))$ we have
\begin{equation}\begin{aligned}\label{eq:widetildebdef}
  \set*{\frac { \partial \mathcal K_{(h)}(m(B,b))}{\partial b}[\mu] = 0 \quad \text{ for all } \mu \in L^2((0,L)) } \iff b =  b_{\min}(B).
\end{aligned}\end{equation}
We are now able to compute the variations of $\mathcal K^0_{(h)}$. For fixed $B, M \in L^2((0,L), \R^{3\times3}_{\skewsym})$ we calculate by using the chain rule
\begin{equation}\begin{aligned}\label{eq:Kbarderivativestepone}
 \paren*{\frac{ \partial }{\partial B}\mathcal K^0_{(h)}}(B) [M]
&=
\frac{ \partial }{\partial B} \paren*{ { \mathcal K}_{(h)}\paren[\Big]{m\paren[\big]{B, b_{\min}(B)}} }[M]\\
&=
 \paren*{ \frac{{ \partial\mathcal K}_{(h)}}{\partial m }\paren[\Big]{m\paren[\big]{B, b_{\min}(B)}} }\brackets*{\frac { \partial m(B, b_{\min}(B))}{\partial B}[M]}.
\end{aligned}\end{equation}
From~\eqref{eq:mdefinition} and the linearity of $b_{\min}$ we obtain
\begin{equation*}\begin{aligned}
\frac {\partial m(B, b_{\min}(B))}{\partial B}[M] = M \pro + ((\partial_B  b_{\min})(0): M)e_1
\end{aligned}\end{equation*}
and thus~\eqref{eq:Kbarderivativestepone} can be rewritten to
\begin{equation}\begin{aligned}\label{eq:Kbareq}
\paren*{\frac{ \partial }{\partial B} \mathcal K^0_{(h)}}(B) [M]&=\lim_{h\down 0}\int_\Omega \lina\paren*{\iota(m(B, b_{\min}(B))) + \sym \nabla_h \psi^h_{m(B, b_{\min}(B))}} \\
&\hspace{5cm}: \iota ( M \pro + ((\partial_B  b_{\min})(0): M)e_1).
\end{aligned}\end{equation}
The function $b_{\min}(B)$ satisfies according to~\eqref{eq:widetildebdef} the equation
\begin{equation}\begin{aligned}\label{eq:Kbarbderivative}
  0&=\frac { \partial \mathcal K_{(h)}}{\partial b}(m(B, b_{\min}(B))[\mu]\\
&=\lim_{h\down 0}\int_\Omega \lina\paren*{\iota(m(B, b_{\min}(B))) + \sym \nabla_h \psi^h_{m(B, b_{\min}(B))}} : \iota (  \mu e_1)
\end{aligned}\end{equation}
for all $\mu \in L^2 ((0,L))$. Finally using $\mu := (\partial_B  b_{\min})(0): M$ in~\eqref{eq:Kbarbderivative} allows us to simplify~\eqref{eq:Kbareq} to
\begin{equation}\begin{aligned}\label{eq:BderivativeofK0}
\paren*{\frac{ \partial }{\partial B}  \mathcal K^0_{(h)}}(B) [M]&=\lim_{h\down 0}\int_\Omega \lina\paren*{\iota(m(B, b_{\min}(B))) + \sym \nabla_h \psi^h_{m(B, b_{\min}(B))}} : \iota ( M \pro ).
\end{aligned}\end{equation}

\subsection{Derivation of the limit Euler-Lagrange equation}
Let $(y, d_2, d_3) \in \mathcal A$, where $\admissible$ is give in~\eqref{eq:admissibleset}, and assume, in addition, that $y(0) =0$. The associated rotation function is then given by $R = (y', d_2, d_3) \in W^{1,2}((0,L),\SO(3))$. Recall that the limit energy of $(y,d_2,d_3)$ is
\begin{equation*}\begin{aligned}
  \energyfunction^0( y, d_2, d_3 ) &  = \mathcal K^0_{(h)}( R^T R') - \int_0^L g \cdot y \,\td x_1 \\
                                 &  = \mathcal K^0_{(h)}( R^T R') - \int_0^L \widehat g \cdot y' \td x_1,
\end{aligned}\end{equation*}
with $\widehat g(x_1) = \int_{x_1}^L g(s) \td s$. We say $(y, d_2, d_3)$ is a \textit{stationary point} of $\energyfunction^0$, if for any $C^1$-curve $\gamma: (-\infty, \infty) \to \mathcal A$ with $\gamma(0) = (y,d_2, d_3)$ we have
\[
  \restrict{ \paren*{\partial_{\eps} \energyfunction^0 [ \gamma(\eps) ]} }{\eps = 0} = D \energyfunction^0[y, d_2, d_3][ \dot \gamma(0)] = 0,
\]
where $\dot \gamma$ denotes the derivative of $\gamma$. The following lemma gives an alternative characterization by identifying the tangent spaces of $\mathcal A$ and explicitly computing the derivative $D\energyfunction^0$.
\begin{lemma}\label{lem:stationarylimit}
Let $(y, d_2, d_3) \in \mathcal A$. Define $R = (y', d_2, d_3)$ and $A = R^T R'$. Then $(y,d_2,d_3)$ is a stationary point of $\energyfunction^0$ iff.\ for every $\Phi \in W^{1,2}((0,L), \R^{3\times3}_{\skewsym})$ we have
\begin{equation}\begin{aligned}\label{eq:limitEL}
 \paren*{ \frac { \partial} { \partial B }\mathcal K^0_{(h)}}(A) [A\Phi -\Phi A + \Phi'] = \int_0^L \widehat g \cdot (R\Phi e_1) \td x_1.
\end{aligned}\end{equation}
\end{lemma}
\begin{proof}
Let $(y^\eps, d_2^\eps, d_3^\eps)_\eps \subset \admissible$ be a $C^1$-curve with $(y^0, d_2^0, d_3^0) = (y, d_2,d_3)$, and define $(R_\eps)_{\eps} \subset W^{1,2}((0,L), \SO(3))$ by $R_\eps = ((y^\eps)', d_2^\eps, d_3^\eps)$; especially we have $R_0 = R$.  It is well-known that the tangent space of $\SO(3)$ in $R$ is given by $\set{ R\Phi: \Phi \in \R^{3\times3}_{\skewsym}}$. Thus, denoting the derivative of $(R_\eps)$ with respect to $\eps$ by $(\dot{R}_\eps)$, we obtain $R_0^T\dot{R}_0 = \Phi$ for some $\Phi \in W^{1,2}((0,L),\mr^{3\times3}_{\skewsym})$ and the tangent space of $\admissible$ in $(y,d_2,d_3)$ is given by
\begin{equation*}\begin{aligned}
&T_{(y,d_2,d_3)} \admissible = \set[\Big]{(v_1, v_2, v_3) \in W^{2,2}((0,L),\R^3)\times W^{1,2}((0,L),\R^3)^2: \\
&\hspace{5cm}  (y', d_2, d_3)^T(v_1', v_2, v_3) \in W^{1,2}((0,L), \R^{3\times3}_{\skewsym})}.
\end{aligned}\end{equation*}
With the chain rule we obtain
\begin{equation*}\begin{aligned}
  \restrict{\partial_{\eps }\energyfunction^0[y^\eps, d^\eps_2, d^\eps_3]}{\eps=0} &= \restrict{\partial_{\eps} \paren*{\mathcal K^0_{(h)}( R^T_\eps R'_\eps) - \int_0^L \widehat g \cdot y'_\eps}}{\eps = 0}\\
                                                                &= \restrict{\frac{\partial \mathcal K^0_{(h)}}{\partial B}(R^TR')[\partial_\eps( R^T_\eps R'_\eps)]}{\eps=0} - \int_0^L \widehat g \cdot (R_0 \Phi e_1)\\
                                                                &=  \frac{\partial \mathcal K^0_{(h)}}{\partial B}(R^TR')[\dot{R}^T_0 R'_0 + R^T_0 \dot{R}'_0] - \int_0^L \widehat g \cdot (R_0 \Phi e_1).
\end{aligned}\end{equation*}
By using the relationship $R_0^T\dot{R}_0 = \Phi$ we obtain
\[
  \dot{R}^T_0 R'_0 = -\Phi R^T R' \quad \text{ and } \quad R^T_0 \dot{R}'_0 = R^T(R\Phi)' = R^T R' \Phi + R^T R \Phi',
\]
and thus
\[
  \dot{R}^T_0 R'_0 + R^T_0 \dot{R}'_0 = -\Phi R^TR' + R^T R' \Phi + \Phi'.
\]
Furthermore we can insert $A = R^T R'$ into the equality, which finally reads
\[
\dot{R}^T_0 R'_0 + R^T_0 \dot{R}'_0 = -\Phi A+ A \Phi + \Phi'.
\]
Defining 
\[
  x_1 \mapsto \widehat \Phi(x_1) = \paren*{\int_0^{x_1} R(s)\Phi(s)e_1 \td s,\; (R\Phi)(x_1)e_2,\; (R\Phi)(x_1)e_3} \in T_{(y, d_2, d_3)}\admissible
\]
we thus get
\begin{equation*}\begin{aligned}
 (D \energyfunction^0)(y, d_2, d_3)[\widehat \Phi]
=\frac{\partial \mathcal K^0_{(h)}}{\partial B}(A)[A\Phi - \Phi A + \Phi'] - \int_0^L \widehat g \cdot (R \Phi e_1).
\end{aligned}\end{equation*}
If $(y, d_2, d_3)$ is stationary, we left-hand side vanishes and we obtain as claimed
\begin{equation*}\begin{aligned}
\frac{\partial \mathcal K^0_{(h)}}{\partial B}(A)[A\Phi - \Phi A + \Phi'] = \int_0^L \widehat g \cdot (R \Phi e_1).
\end{aligned}\qedhere\end{equation*}
\end{proof}

\section{Proof of the main theorem}
We dedicate the whole section to the proof of Theorem~\ref{thm:main}. From now on let $W^h, y^h, g$ be as in Theorem~\ref{thm:main}.
From the energy bound~\eqref{eq:energyboundinthm} together with the non-degeneracy hypothesis~\ref{(M2)} on $W^h$ we obtain the inequality
\[
  \limsup_{h\down 0} \norm{ \dist(\nabla_h y^h, \SO(3))}_{L^2} < \infty,
\]
and furthermore by assumption on $(y^h)$ we have that $y^h(0, x_2, x_3) = hx_2 e_2 + h x_3 e_3$.  
Thus we may apply Proposition~\ref{thm:compactness} and deduce that there exists a sequence of rotations $(R^h) \subset C^\infty([0,L], \SO(3))$ with properties~\eqref{eq:estimatedifferenceutorot},~\eqref{eq:estimatederivativeR} and~\eqref{eq:Rbdry}.

 We recall the definition of the linearized strain $G^h$ given by
\[
  G^h = \frac { (R^h)^T \nabla_h y^h - \id_{3\times3}} h.
\]
It was already introduced in~\eqref{eq:Ghdefinitionu} and, by the discussion following the definition, there exist a subsequence (not relabeled) and a function $G \in L^2(\Omega, \R^{3\times3})$ such that $G^h \weakly G$ in $L^2$.
From the frame indifference of $W^h$ it follows that
\[
  DW^h(x,F) = R DW^h(x, R^T F)  \quad \text{ for all } F \in \mr^{3\times 3},\, R\in \SO(3), \text{ a.e. } x \in \Omega.
\]
Thus
\begin{equation}\begin{aligned}\label{eq:connectionDWE}
DW^h(x, \nabla_h y^h) = R^h DW^h(x, \id_{3\times 3} + h G^h) = h R^h E^h,
  \end{aligned}\end{equation}
where $E^h:= h^{-1} DW^h(\cdot, \id_{3\times3} + h G^h)$ is the nonlinear stress. On the other hand a Taylor expansion around the identity yields
\[
  DW^h(x, \id_{3\times 3} + hG^h) = h D^2W^h(x, \id_{3\times 3}) G^h + \zeta^h(x, hG^h).
\]
where (S3) implies the estimate
\[
 \abs{ \zeta^h(x, F)} \leq \widehat r(\abs F) \abs F,
\]
for some monotone $\widehat r: [0,\infty) \to [0,\infty)$ with $\widehat r(\eps) \down 0$ if $\eps \down 0$.
 Together with $D^2W^h(\cdot, \id_{3\times 3}) = D^2Q^h(\cdot, 0) = \lina$ we get
 \begin{equation}\begin{aligned}\label{eq:Ehdecomp}
 E^h = \lina\sym G^h + \frac 1 h  \zeta^h(\cdot, hG^h).
 \end{aligned}\end{equation}
The error term $h^{-1}\zeta^h(\cdot, hG^h)$ does not necessarily converge strongly to $0$ in $L^2$,
since $G^h$ might concentrate in $L^2$. We will now show that the error term does not oscillates, and that it weakly converges to zero:
\begin{lemma}\label{lem:weakvanishingofzeta}
Let $(\eta^h) \subset L^2(\Omega)$ be such that $(\abs {\eta^h}^2)$ is equi-integrable. Then
\[
\lim_{h\down 0}\int_\Omega \eta^h \cdot \paren*{\frac 1 h  \zeta^h(\cdot, hG^h)} \td x = 0.
\]
\end{lemma}
Note that this immediately implies $h^{-1} \zeta^h(\cdot, hG^h)\weakly 0$ in $L^2(\Omega, \R^{3\times3})$, and especially that $(h^{-1} \zeta^h(\cdot, hG^h))$ is uniformly bounded in $L^2$.

\begin{proof}
Let $0 < \alpha < 1$. We define the sets $S^\alpha_h = \set { x \in \Omega: \quad h\abs{G^h(x)} \leq h^\alpha }$,
and the truncated function $\widehat G^h := G^h \charact{S_h^\alpha}$.
Obviously $h\widehat G^h \to 0$ in $L^\infty$, $G^h = \widehat G^h$ on $S_h^\alpha$, and by Chebyshev inequality we have 
$\Lebesguee{ \Omega \setminus S_h^\alpha } \to 0$ for $h\down 0$.
We can now compute 
\begin{align*}	
  \norm*{ \frac 1 h  \zeta^h(\cdot, h\widehat G^h)}_{L^2}^2 &= \frac 1 {h^2} \int_{\Omega}\abs*{\zeta^h(x, h\widehat G^h)}^2 \td x\\
                                                          &\leq \frac 1 {h^2} \int_{\Omega} \widehat r(\norm{h\widehat G^h}_\infty)^2 \abs{h\widehat G^h}^2 \td x\\
                                                          &\leq \widehat r(\norm{h\widehat G^h}_\infty)^2  \norm{\widehat G^h}_{L^2}^2 \leq \widehat r(\norm{h\widehat G^h}_\infty)^2  \norm{G^h}_{L^2}^2\to 0,
\end{align*}
by the uniform bound of $ G^h$ in the $L^2$-norm. Finally applying Hölder's inequality yields
\[
\abs*{
  \int_\Omega \eta^h \cdot \paren*{\frac 1 h \paren[\Big]{ \zeta^h(x,hG^h) - \zeta^h(x,h\widehat G^h) } }\td x
}
\leq C\int_{S_h} \abs{\eta^h}^2\td x \to 0,
\]
which implies the claim.
\end{proof}
With this result we can deduce the limit PDE in terms of the stress. The part follows closely the corresponding proof in~\cite{MM08}, and thus we skip some details.

\vspace{\baselineskip}

{\bfseries Compactness}

From the properties~\eqref{eq:estimatedifferenceutorot}--\eqref{eq:Rbdry} for the sequence $(R^h)$,
we deduce that there exist a subsequence (not relabeled) and limit $R \subset W^{1,2}((0,L), \SO(3))$ such that $R(0) = \id_{3\times3}$ and $R^h \weakly R$ in $W^{1,2}((0,L), \SO(3))$. Defining $\overline y(x_1) = \int_0^{x_1} R(s) e_1 \td s$, $\overline d_k= R e_k$ for $k=2,3$ we obtain $\overline y \in W^{2,2}_{\bdy}([0,L],\R^3)$, $y^h \to \overline y$ strongly in $W^{1,2}(\Omega, \R^3)$, $\nabla_h y^h \to (\overline y', \overline d_2, \overline d_3)$ strongly in $L^2(\Omega, \R^{3\times3})$, $\overline d_k(0) = e_k$ for $k=2,3$ and $(\overline y, \overline d_2, \overline d_3) \in \admissible$.

\vspace{\baselineskip}

{\bfseries Properties of $E^h$}

We start by using the uniform energy bound of the deformations $(y^h)$, i.e., stationary is not yet needed. Recall the decomposition~\eqref{eq:Ehdecomp}, i.e., 
\[
 E^h = \lina(x) \sym G^h + \frac 1 h  \zeta^h(x, hG^h),
\]
Notice that the uniform bound on $\abs { \lina } \leq C\beta$ given by~\eqref{eq:Qproperties}, the uniform $L^2$ bound on $G^h$ and the uniform $L^2$ bound on the sequence $(h^{-1} \zeta^h(\cdot, hG^h))_{h > 0}$, following from Lemma~\ref{lem:weakvanishingofzeta}, imply 
a uniform $L^2$ bound on the sequence $E^h$. Thus $E^h$ weakly subconverges to some $E \in L^2( \Omega, \R^{3\times3})$.
The frame-indifference~\ref{(M1)} readily implies that $DW^h(\cdot, F)F^T$ is symmetric for every $F \in \R^{3\times3}$ almost everywhere on $\Omega$.
For $F = \id_{3\times3} + hG^h$ the statement $\skewsym ( DW^h(\cdot, F)F^T) = 0$ can be rewritten to
\begin{equation}\begin{aligned}\label{eq:skewsymE}
  \skewsym ( E^h ) = h \skewsym ( G^h (E^h)^T ).
\end{aligned}\end{equation}
From the uniform $L^2$ bound on $E^h$ and $G^h$ we deduce a uniform $L^1$ bound on $( h^{-1} \skewsym (E^h))$.

\vspace{\baselineskip}

{\bfseries Deriving Euler-Lagrange equations}

Since $(y^h)$ are stationary points of $\energyfunction^h$ we obtain for any $\psi \in C^\infty_{\bdy}(\overline \Omega,\R^3)$ the equality
\begin{equation*}\begin{aligned}
\int_\Omega \paren*{DW^h(x, \nabla_h y^h(x)): \nabla_h \psi(x)- h^2 g(x_1)\cdot \psi(x) }\td x = 0.
\end{aligned}\end{equation*}
By density the equation also holds for arbitrary $\psi \in W^{1,2}_{\bdy}(\Omega,\R^3)$.
Using~\eqref{eq:connectionDWE} we rewrite this equation to
\begin{equation}\begin{aligned}\label{eq:fullEL}
\int_\Omega\paren*{ R^h E^h: \nabla_h \psi - h g \cdot \psi }\td x = 0.
\end{aligned}\end{equation}

For $\psi(x) = \varphi(x_1)$ with $\varphi \in C^\infty_{\bdy}([0,L], \R^3)$ the equation~\eqref{eq:fullEL} reduces to
\begin{equation}\label{eq:constEL}
 \int_0^L\int_\omega \paren*{R^h  E^h e_1 \cdot \varphi' - hg \cdot \varphi}\td x' \td x_1 = \int_0^L \paren*{R^h \overline E^h e_1 \cdot \varphi' - hg \cdot \varphi} \td x_1 = 0,
\end{equation}
where $\overline E^h(x_1):= \int_\omega E^h(x_1, x')\td x' \in L^2((0,L), \R^{3\times3})$. Furthermore we denote the first moments with respect to $x_2$ and $x_3$ of $E$ by $\widetilde E, \widehat E \in L^2((0,L), \R^{3\times3})$ respectively; more precisely let
\begin{equation*}\begin{aligned}
\widetilde E(x_1)= \int_\omega x_2 E(x_1, x') \td x'; \quad
\widehat E(x_1)= \int_\omega x_3 E(x_1, x') \td x'.
\end{aligned}\end{equation*}

Let $\phi \in C^\infty_{\bdy}([0,L])$. Then for $\psi(x) = x_2 \phi(x_1) R^h(x_1) e_1 \in W^{1,2}_{\bdy}(\Omega, \R^3)$ we obtain 
\[
  \nabla_h \psi(x) = \begin{pmatrix} x_2 \phi'(x_1) R^h (x_1)e_1 + x_2 \phi(x_1) (R^h)'(x_1)e_1 \;\Big|\; \frac 1 h \phi(x_1) R^h(x_1) e_1 \;\Big|\; 0\end{pmatrix}
\]
and thus~\eqref{eq:fullEL} simplifies to
\begin{align*}	
0 &= \int_\Omega \paren*{R^h E^h: \nabla_h \psi - h g \cdot \psi} \td x \\
  &=\int_0^L \paren*{R^h \widetilde E^he_1 \cdot \phi' R^h e_1 + R^h \widetilde E^he_1\cdot \phi (R^h)'e_1 + \frac 1 h R^h \overline E^h e_2 \cdot \phi R^h e_1} \td x_1.
\end{align*}
Introducing $A^h := (R^h)^T (R^h)'$ this simplifies further to
\begin{equation}\begin{aligned}\label{eq:x2EL}
  \int_0^L \paren*{ \widetilde E^h_{11} \cdot \phi' + \phi \widetilde E^h e_1 \cdot A^h e_1 + \phi \frac 1h \overline E_{12}^h} \td x_1= 0.
\end{aligned}\end{equation}
Analogously for $\psi(x) = x_3 \phi(x_1) R^h(x_1) e_1$ we get
\begin{equation}\begin{aligned}\label{eq:x3EL}
  \int_0^L  \paren*{\widehat E^h_{11} \cdot \phi' + \phi \widehat E^h e_1 \cdot A^h e_1 + \phi \frac 1h \overline E_{13}^h } \td x_1 = 0,
\end{aligned}\end{equation}
and finally $\psi(x) = x_3\phi(x_1)  R^h(x_1) e_2 - x_2\phi(x_1)  R^h(x_1) e_3$ yields
\begin{equation}\begin{aligned}\label{eq:x2x3EL}
  \int_0^L \paren*{\phi' ( \widehat E_{21}^h - \widetilde E_{31}^h ) + \phi ( \widehat E^h e_1 \cdot A^h e_2 - \widetilde E^h e_1 \cdot A^h e_3) + \phi \frac 1 h ( \overline E_{23}^h - \overline E_{32}^h)}\td x_1 =0.
\end{aligned}\end{equation}

\vspace{\baselineskip}

{\bfseries Consequences of the Euler-Lagrange equations}

Now, by stationary of $(y^h)$, the equation~\eqref{eq:constEL} holds for arbitrary $\varphi \in C^\infty_c( (0,L), \R^3)$, and thus
\begin{equation}\begin{aligned}\label{eq:constELptws}
  \overline E^h e_1 = - h{( R^h)}^T \widehat g \quad \text{a.e.\ in } (0,L),
\end{aligned}\end{equation}
especially
\begin{equation}\begin{aligned}\label{eq:constELptwslim}
  \overline E e_1 = 0 \quad \text{a.e.\ in } (0,L).
\end{aligned}\end{equation}
Furthermore the equations~\eqref{eq:x2EL},~\eqref{eq:x3EL} and~\eqref{eq:x2x3EL} imply that
$\widetilde E^h_{11}$, $\widehat E^h_{11}$ and $(\widehat E^h_{21} - \widetilde E^h_{31})$ are weakly differentiable.
The respective derivatives are in $L^1$, as seen by combining~\eqref{eq:skewsymE},~\eqref{eq:constELptws} together with the uniform $L^2$ bound on $A^h$, which was just $(R^h)^T (R^h)'$. 
By Sobolev's Embedding Theorem we thus obtain that
\begin{equation}\begin{aligned}\label{eq:muellersstrongconvergence}
(\widetilde E_{11}^h), (\widehat E_{11}^h), ( \widehat E_{21}^h - \widetilde E_{31}^h ) \text{  converge strongly in } L^2((0,L)).
  \end{aligned}\end{equation}
From this we immediately get the following:
Let $(M^h) \subset L^2( (0,L), \R^{3\times3}_{\skewsym})$ with $M^h \weakly M \in L^2( (0,L), \R^{3\times3}_{\skewsym})$. Then by direct calculation we obtain
\[
  \int_\Omega E^h : \iota ( M^h \pro)  \td x = \int_0^L ( \widetilde E_{31}^h - \widehat E_{21}^h, \widehat E_{11}^h, \widetilde E_{11}^h ) \cdot \axl M^h \td x_1
\]
and thus by applying~\eqref{eq:muellersstrongconvergence} we get
\begin{equation}\begin{aligned}\label{eq:Etimesiotaskew}
\lim_{h\down 0} \int_\Omega E^h : \iota ( M^h \pro)  \td x
= \lim_{h\down 0}\int_\Omega E^h : \iota ( M \pro)  \td x.
\end{aligned}\end{equation}

\vspace{\baselineskip}

{\bfseries The limit of the PDE in terms of the stress}

Fix some $\Phi \in C^\infty_{\bdy}( [0, L], \R^{3\times3}_{\skewsym})$ and let $\phi_1,\phi_2, \phi_3$ be given by $\axl(\Phi) = (\phi_1, \phi_2, \phi_3)$. 
We then define the test functions
\begin{align*}	
  \psi^h(x_1, x_2, x_3) &= R^h(x_1) \Phi(x_1)\pro(x) \\
  & \paren[\Big]{= x_3 \phi_2 R^h e_1 -x_2 \phi_3 R^h e_1 + \phi_1 ( x_2 R^h e_3 - x_3 R^h e_2 )}.
\end{align*}
We compute
\begin{align*}	
\nabla_h \psi^h = \begin{pmatrix}   R^h \Phi' \pro +(R^h)' \Phi \pro
                        \;\Big|\; \frac 1 h R^h \Phi e_2
                        \;\Big|\; \frac 1 h R^h \Phi e_3
 \end{pmatrix},
\end{align*}
and plugging it into~\eqref{eq:fullEL} we obtain
\begin{equation}\begin{aligned}\label{eq:afterinsertingtestfunction}
&\int_\Omega \paren*{R^h E^h: \nabla_h \psi^h - h  g \cdot  \psi^h }\td x   \\
&= \int_\Omega \paren*{ E^h e_1\cdot \Phi' \pro + E^h e_1 \cdot A^h \Phi \pro + \frac 1 h E^h e_2 \cdot \Phi e_2 + \frac 1 h E^h e_3 \cdot \Phi e_3}\td x.
\end{aligned}\end{equation}
By definition we have $\Phi e_2 = \phi_1 e_3 - \phi_3 e_1$ and $\Phi e_3 = \phi_2 e_1 - \phi_1 e_2$ and thus
\begin{align*}	
 E^h e_2 \cdot \Phi e_2 +  E^h e_3 \cdot \Phi e_3
&= ( \phi_1 E^h_{32} - \phi_3 E^h_{12} ) + (\phi_2 E^h_{13}  - \phi_1 E^h_{23})\\
&= 2\phi_1 (\skewsym E^h)_{32}  - 2\phi_{3} (\skewsym E^h)_{12} + 2\phi_2 (\skewsym E^h)_{13}\\
&\quad - \phi_3 E^h_{21} + \phi_2 E^h_{31} \\
&= (\Phi : \skewsym E^h) + ( \phi_2 E^h_{31} - \phi_3 E^h_{21})
\end{align*}
With the preceding calculation it is easy to verify the splitting of~\eqref{eq:afterinsertingtestfunction} into 
\begin{equation}\begin{aligned}\label{eq:PDEsplit}
\int_\Omega \paren*{R^h E^h: \nabla_h \psi^h - h  g \cdot  \psi^h }\td x  = \I^h + \II^h + \III^h,
\end{aligned}\end{equation}
where
\begin{align}
  \label{eq:Idef}\I^h     &:= \int_\Omega  \paren*{E^h e_1 \cdot \Phi' \pro }\td x = \int_\Omega E^h: \iota( \Phi' \pro) \td x, \\
  \notag \II^h   &:= \int_\Omega \paren*{\phi_2 \frac 1h  E_{31}^h - \phi_3 \frac 1h  E_{21}^h} \td x, \\
  \label{eq:IIIdef}\III^h &:= \int_\Omega \paren*{E^h e_1 \cdot A^h \Phi \pro + \frac 1 h \Phi: \skewsym E^h }\td x.
\end{align}
The third one will be the most difficult to handle.

\vspace{\baselineskip}

{\bfseries Regarding $\II^h$}, from~\eqref{eq:constELptws} we obtain $\overline E^h e_1 = -h (R^h)^T \widehat g$ and thus
\begin{equation}\begin{aligned}\label{eq:IIident}
 \II^h = \int_0^L \paren*{\phi_3 \widehat g\cdot (R^he_2) - \phi_2 \widehat g\cdot  (R^he_3)} \td x_1  =\int_0^L \widehat g \cdot (R^h \Phi e_1)\td x_1.
\end{aligned}\end{equation}

{\bfseries Regarding $\III^h$}, we claim that we have 
\begin{equation}\begin{aligned}\label{eq:IIIident}
\lim_{h\down 0}\III^h  &= \lim_{h\down 0}\int_\Omega \paren*{E^h e_1 \cdot (A\Phi  - \Phi A)\pro} \td x.
\end{aligned}\end{equation}

Indeed, recall that from~\eqref{eq:Ghdefinitionu} we have
\begin{equation*}\begin{aligned}
G^h = A^h \pro \otimes e_1+ ( R^h)^T \nabla_h  z^h,
\end{aligned}\end{equation*}
where $z^h$ was defined by~\eqref{eq:zhintro}. By making use of~\eqref{eq:skewsymE} we obtain
\begin{align*}	
\frac 1 h\skewsym (E^h)& =  \skewsym( G^h (E^h)^T )\\ 
&= \skewsym\paren*{\paren[\big]{( R^h)^T \nabla_h  z^h+   A^h \pro \otimes e_1} (E^h)^T}\\
&= \skewsym\paren*{(R^h)^T \nabla_h  z^h (E^h)^T} + \skewsym\paren*{ (A^h \pro \otimes e_1 )(E^h)^T}.
\end{align*}
Furthermore by the skew-symmetry of $\Phi$ we thus obtain
\begin{align}	\label{eq:PhitimesskewE}
  \frac 1 h\Phi: \skewsym(E^h) = \Phi: \paren*{(R^h)^T \nabla_h  z^h (E^h)^T} + \Phi:\paren*{ A^h \pro \otimes E^he_1}.
\end{align}
Note that for any $M \in \R^{n\times n}$ and $v,w \in \R^{n}$ we have the algebraic identity
\[
  M : (v \otimes w) = \tr ( M^T (v\otimes w)) =  \tr( (M^T v) \otimes w) =  (M^T v) \cdot w,
\]
which applied to $M =\Phi, v = A^h\pro, w = E^h e_1$ yields for the second term in~\eqref{eq:PhitimesskewE} the equality
\[
  \Phi: (A^h\pro \otimes E^h e_1) =  - E^h e_1 \cdot (\Phi A^h \pro).
\]
With this we can simplify $\III^h$, given by~\eqref{eq:IIIdef}, to
\begin{equation}\begin{aligned}\label{eq:IIIlevelhoptimal}
\III^h & = \int_\Omega \paren*{E^h e_1 \cdot A^h \Phi \pro + \frac 1 h \Phi: \skewsym E^h }\td x \\
           & = \int_\Omega \paren*{E^h e_1 \cdot (A^h \Phi - \Phi A^h)\pro +\Phi: ((R^h)^T \nabla_h  z^h (E^h)^T) }\td x.
\end{aligned}\end{equation}
We start by proving that,
\begin{equation*}\begin{aligned}
\lim_{h\down 0}\int_\Omega \paren*{E^h e_1 \cdot (A^h \Phi - \Phi A^h)\pro}\td x
=\lim_{h\down 0}\int_\Omega \paren*{E^h e_1 \cdot (A \Phi - \Phi A)\pro}\td x.
\end{aligned}\end{equation*}
which, however, immediately follows from~\eqref{eq:Etimesiotaskew} by setting $M^h :=  A^h \Phi - \Phi A^h$ and $M:= A\Phi - \Phi A$.
If the second term on the right-hand side of~\eqref{eq:IIIlevelhoptimal} vanishes, the claim is proved.

For this we first note that for any $B \in \R^{3\times3}$ and $R \in \SO(3)$ we have $\skewsym B = \skewsym(R B R^T)$. This is a straight-forward computation,
which relies heavily on the fact, that $3\times 3$ skew-symmetric matrices have at most two non-vanishing entries per column and row. With $R := R^h$ and $B :=\nabla_h  z^h (E^h)^T$ we then obtain
\[
\skewsym\paren*{(R^h)^T \nabla_h  z^h (E^h)^T} = \skewsym\paren*{ \nabla_h  z^h (E^h)^T(R^h)^T} =  \skewsym \paren* {\nabla_h  z^h (R^h E^h)^T },
\]
and thus
\begin{equation}\begin{aligned}\label{eq:zerrorterm}
\int_\Omega {\Phi : \paren* {\nabla_h  z^h (R^h E^h)^T }} \td x 
&=\int_\Omega {\paren*{(\nabla_h  z^h)^T\Phi} : \paren* { (R^h E^h)^T }} \td x \\
&= -\int_\Omega   R^h E^h  :\paren*{\Phi \nabla_h  z^h}\td x.
    \end{aligned}\end{equation}
We write the right-hand side of the inner product as a gradient and a lower-order term, i.e.,
$\Phi  \nabla_h  z^h = \nabla_h( \Phi  z^h) - \iota(\Phi'  z^h)$.
Using this identity in~\eqref{eq:zerrorterm} we obtain two terms. For the first one we note that $\Phi$ vanishes at the left boundary, and we might use the Euler-Lagrange equation~\eqref{eq:fullEL} to get
\[
\int_\Omega\paren[\Big]{   R^h E^h  :\nabla_h( \Phi  z^h)}\td x = h \int_\Omega g \cdot ( \Phi  z^h) \td x \to 0.
\]
For the second term we use the strong convergence of $z^h$ and $R^h$ to go to the limit
\begin{equation*}\begin{aligned}
\lim_{h\down0}\int_\Omega   \paren[\Big]{R^h E^h  :\iota( \Phi ' z^h)}\td x = \int_\Omega \paren[\Big]{R Ee_1 \cdot(\Phi'  z) }\td x
                                                              = \int_0^L \paren[\Big]{R \overline Ee_1 \cdot(\Phi'  z)}\td x_1,
\end{aligned}\end{equation*}
where in the last step we used that $R$ and $z$ are independent of $x_2, x_3$. Since $\overline Ee_1 = 0$ by~\eqref{eq:constELptwslim} this term vanishes as well, and the claim~\eqref{eq:IIIident} is thus proved.
\vspace{\baselineskip}

Inserting~\eqref{eq:Idef},~\eqref{eq:IIident} and~\eqref{eq:IIIident} into~\eqref{eq:PDEsplit} we obtain
\begin{equation}\begin{aligned}\label{eq:almostlimit}
\lim_{h\down 0}\int_\Omega \paren*{R^h E^h: \nabla_h \psi - h  g \cdot \psi }\td x   
&=
\lim_{h\down 0}\int_\Omega E^h: \iota( (A\Phi  - \Phi A)\pro + \Phi' \pro) \td x\\
    &\quad+ \lim_{h\down 0}\int_0^L \widehat g \cdot (R^h \Phi e_1)\td x_1.
    \end{aligned}\end{equation}
From the strong convergence $R^h \to R$ in $L^\infty$ we obtain for the second term
\begin{equation}\begin{aligned}\label{eq:limitforces}
\lim_{h\down 0}\int_0^L \widehat g \cdot (R^h \Phi e_1)\td x_1 = \int_0^L \widehat g \cdot (R \Phi e_1)\td x_1,
  \end{aligned}\end{equation}
while for the first term we will show that
\begin{equation}\begin{aligned}\label{eq:lastbendingconvergence}
\lim_{h\down 0}\int_\Omega E^h: \iota( (A\Phi  - \Phi A)\pro + \Phi' \pro) \td x = \frac { \partial  {\mathcal K}_{(h)}}{ \partial m}(m_d)[(A\Phi - \Phi A + \Phi')\pro].
\end{aligned}\end{equation}

\vspace{\baselineskip}

{\bfseries Identification of the limit}

To show~\eqref{eq:lastbendingconvergence} we will first prove the analogue to~\cite[lemma~3.1]{BVP17}, whose approach we will follow from now on.
\begin{lemma}\label{lem:orthequi}
Let $(u^h) \subset W^{1,2}(\Omega, \R^3)$ be such that $\twist(u^h_2, u^h_3) \to 0$ strongly in $L^2$,
  $\paren*{\abs{\sym \nabla_h u^h}^2}$ is equi-integrable and $(u_1^h, h u_2^h, h u_3^h) \to 0$ strongly in $L^2$.
Then for all $ \phi \in C^\infty_{\bdy}( [0,L] )$ we have
\begin{equation}\begin{aligned}
\label{eq:orthwitherror}
\lim_{h\down 0} \int_\Omega \paren[\Big]{\phi \lina G^h: \nabla_h u^h} \td x
&= 0. 
\end{aligned}\end{equation}
\end{lemma}
\begin{proof}
Fix some $\phi \in C^\infty_{\bdy}( [0,L] )$ and let $(u^h)$ be as assumed in the lemma. By Prop.~\ref{prop:grisoa} there exists a constant $C_\omega > 0$, depending only on $\omega$, and sequences $(B^h) \subset W^{1,2}((0,L), \R^{3\times3}_{\skewsym})$, $(\vartheta^h) \subset W^{1,2}(\Omega,\R^3)$ and $(o^h) \subset L^2(\Omega, \R^{3\times 3})$ with
\begin{equation}\begin{aligned}\label{eq:orthudecomp}
 \sym \nabla_h u^h = \sym \iota ( (B^h)' \pro) + \sym \nabla_h \vartheta^h + o^h,
\end{aligned}\end{equation}
that, in addition, satisfy the bounds
\[
  \norm{ B^h }_{W^{1,2}} + \norm {\vartheta^h}_{L^2} + \norm{{ \nabla_h \vartheta^h }}_{L^2} \leq C_\omega \norm { \sym \nabla_h u^h}_{L^2}.
\]
Furthermore $B^h,\vartheta^h, o^h\to0$ strongly in $L^2$, and $(\abs{(B^h)'}^2)$, $(\abs { \nabla_h \vartheta^h}^2)$ are both equi-integrable.
Using~\eqref{eq:Ehdecomp} we can write~\eqref{eq:orthwitherror} as
\begin{equation}\begin{aligned}\label{eq:convsymuh}
\int_\Omega \paren[\Big]{\phi \lina \sym G^h: \sym (\nabla_h u^h)} \td x &= \int_\Omega \paren[\Big]{\phi E^h: \sym (\nabla_h u^h)}\td x \\
   &\quad- \frac 1 h\int_\Omega \paren[\Big]{\phi  \zeta^h(x,hG^h) : \sym (\nabla_h u^h)}\td x.
\end{aligned}\end{equation}
 The first term on the right-hand side can be decomposed with~\eqref{eq:orthudecomp} to
 \begin{equation}\begin{aligned}\label{eq:orthuinserted}
 \int_\Omega \paren*{\phi E^h: \sym (\nabla_h u^h) }\td x 
   = \int_\Omega \paren[\Big]{\phi E^h: \sym \paren[\big]{\iota ((B^h)' \pro ) +  \nabla_h \vartheta^h + o^h} }\td x.
   \end{aligned}\end{equation}
Clearly the term containing $o^h$ vanishes in the limit. 
By symmetry of $\lina$ we have $\skewsym E^h = \frac 1 h \skewsym \zeta^h(\cdot, hG^h)$ and thus write
\begin{equation}\begin{aligned}\label{eq:skeweiszetausage}
&\int_\Omega \paren[\Big]{\phi E^h: \sym \paren[\big]{\iota ((B^h)' \pro ) +  \nabla_h \vartheta^h } }\td x\\
&\quad=\int_\Omega \paren[\Big]{\phi E^h:  \paren[\big]{\iota ((B^h)' \pro ) +  \nabla_h \vartheta^h } }\td x \\
  &\qquad-\int_\Omega \paren[\Big]{\phi \frac 1 h \zeta^h(x, hG^h): \skewsym \paren[\big]{\iota ((B^h)' \pro ) +  \nabla_h \vartheta^h } }\td x.
  \end{aligned}\end{equation}

Combining~\eqref{eq:orthuinserted} with~\eqref{eq:skeweiszetausage} yields
\begin{equation*}\begin{aligned}
\lim_{h\down 0}\int_\Omega \paren[\Big]{\phi \lina G^h: \sym (\nabla_h u^h) }\td x
 &= \lim_{h \down 0} \int_\Omega \paren[\Big]{\phi E^h:  \paren[\big]{\iota ((B^h)' \pro ) +  \nabla_h \vartheta^h } }\td x \\
  &-\lim_{h\down 0}\int_\Omega \phi \frac 1 h \zeta^h(x, hG^h):  \paren[\Big]{\iota ((B^h)' \pro ) +  \nabla_h \vartheta^h } \td x.
\end{aligned}\end{equation*}

We start with the first term on the right-hand side, i.e.,
\[
  \int_\Omega \paren[\Big]{\phi E^h:  \iota ((B^h)' \pro ) }\td x.
\]
By applying~\eqref{eq:Etimesiotaskew} with $M^h = (B^h)'$ and $M =0$ we see that this term vanishes in the limit. Secondly we study 
\[
\int_\Omega\paren[\Big]{ \phi E^h:  \nabla_h \vartheta^h} \td x,
\]
and for this we rewrite the term to
\begin{equation}\begin{aligned}\label{eq:partoflemma32}
\int_\Omega\paren[\Big]{ \phi E^h:  \nabla_h \vartheta^h} \td x
&=\int_\Omega \paren[\Big]{\phi R^h E^h: R^h \nabla_h \vartheta^h} \td x \\
&= \int_\Omega \paren[\Big]{R^hE^h: \nabla_h (R^h\phi \vartheta^h) }\td x -\int_\Omega\paren[\Big]{ R^hE^h:  \iota((R^h\phi)' \vartheta^h ) }\td x.
\end{aligned}\end{equation}
For the first term on the right-hand side we use the Euler-Lagrange equation and obtain
\[
\int_\Omega E^h: \nabla_h (R^h\phi \vartheta^h) \td x =h \int_\Omega  g \cdot (R^h\phi \vartheta^h) \to 0,
\]
while we split once more the second term on the right-hand side of~\eqref{eq:partoflemma32} into 
\begin{equation}\begin{aligned}\label{eq:orthlesserorder}
\int_\Omega \paren[\Big]{R^h E^h:  \iota((R^h\phi)' \vartheta^h ) }\td x 
& =\int_\Omega \paren[\Big] {R^h E^he_1 \cdot  (R^h\phi)'( \vartheta^h - \overline \vartheta^h) } \td x \\
   &\quad +\int_\Omega \paren[\Big]{R^h E^he_1 \cdot  (R^h\phi)' \overline \vartheta^h } \td x,
   \end{aligned}\end{equation}
where $\overline \vartheta^h(x_1) = \int_\omega \vartheta(x_1, x') \td x$. 
By the uniform bound of $h (R^h)''$ in $L^2$, stated in~\eqref{eq:estimatederivativeR}, we obtain the uniform bound of $h (R^h)'$ in $W^{1,2}$. From the compact Sobolev embedding we obtain that $(h (R^h)')$ is strongly compact in $L^\infty$. Since $(R^h)'$ is bounded in $L^2$, we have $h (R^h)' \to 0$ strongly in $L^2$. By uniqueness of the limit we have $h(R^h)' \to 0$ strongly in $L^\infty$. 
We apply Poincar\'e's inequality and obtain
\[
  \norm { \vartheta^h - \overline \vartheta^h}_{L^2} \leq C \norm { \partial_2 \vartheta^h}_{L^2} \leq C h \norm{ \nabla_h \vartheta^h }_{L^2} \leq C h.
\]
This bound, together with $h (R^h)' \to 0$ strongly in $L^\infty$, implies
\[
h(R^h\phi)'\frac{( \vartheta^h - \overline \vartheta^h)}h \to 0 \quad \text{ strongly in } L^2,
\]
while for the second term in~\eqref{eq:orthlesserorder} we use Sobolev embedding to obtain $\overline \vartheta^h \to 0$ in $L^\infty$. Combining both we conclude the vanishing of~\eqref{eq:orthlesserorder}.
Finally for the last remaining term in~\eqref{eq:convsymuh}, namely
\[
  \lim_{h\down 0}\int_\Omega \phi(x_1) \frac 1 h \zeta^h(x, hG^h):  \paren[\Big]{\iota ((B^h)' \pro ) +  \nabla_h \vartheta^h } \td x,
\] we use that $(\abs{(B^h)'}^2)$ and $(\abs{\nabla_h \vartheta^h}^2)$ are equi-integrable, and thus by virtue of Lemma~\ref{lem:weakvanishingofzeta} this term vanishes as well. This finishes the proof of the lemma.
\end{proof}
We finally prove~\eqref{eq:lastbendingconvergence}. For this we decompose $E^h$ into $E^h = \lina G^h+ \frac 1 h \zeta^h(\cdot, hG^h)$ and apply Lemma~\ref{lem:weakvanishingofzeta} to obtain
\begin{equation}\begin{aligned}\label{eq:mainproofapproxEbyA}
&\lim_{h\down 0}\int_\Omega E^h: \iota( (A\Phi  - \Phi A)\pro + \Phi' \pro) \td x \\
  &\hspace{2cm}= \lim_{h\down 0}\int_\Omega \lina G^h: \iota( (A\Phi  - \Phi A)\pro + \Phi' \pro) \td x.
\end{aligned}\end{equation}
From the decomposition~\eqref{eq:symGhdecomp} we get 
\[
  \sym G^h = \sym \iota(m_d) + \sym \nabla_h \psi^h + o^h 
\]
for the fixed part $m_d \in L^2(\Omega, \R^3)$ and the corrector sequence $\psi^h$ introduced in~\eqref{eq:definitionmd} and~\eqref{eq:definitioncorrector} respectively, and the sequence $o^h$, which converges strongly to zero in $L^2$.

We show that $\sym \nabla_h\psi^h$ and $\sym \nabla_h\psi^h_{m_d}$ are, up to $L^2$-concentration, close in $L^2$, where $(\psi^h_{m_d})$ is the relaxation sequence given by Lemma~\ref{lem:orthabstract}.
Indeed, we first use identity~\eqref{eq:sympsidecomp} to obtain
\[
  \sym \nabla_h \psi^h = \sym \iota ( {(\Psi^h)}'\pro) + \sym \nabla_h v^h,
\]
where $\Psi^h, v^h$ are defined prior to this decomposition.
By applying~\cite[lemma~2.17]{MV16} to $(\Psi^h)$ and $(v^h)$, we obtain a subsequence $(h)$ (not relabeled), a sequence of measurable sets $O^h$ with $\lim_{h\down 0}\Lebesgueee{\Omega \setminus O^h} = 0$ and a sequences $(\widetilde \Psi^h), (\widetilde v^h)$ such that  $( \abs {( \widetilde\Psi^h)'}^2, \abs{(\nabla_h\widetilde v^h)}^2$ are equi-integrable and
\[
  \norm { (\Psi^h - \widetilde \Psi^h)'}_{L^2(O^h)} + \norm { \nabla_h ( v^h - \widetilde v^h)}_{L^2(O^h)}  \to 0.
\] 
By~\eqref{prop:grisob} there exists $(\widetilde \psi^h) \subset W^{1,2}(\Omega,\R^3)$ such that 
\[
  \sym \nabla_h \widetilde \psi^h  = \sym \iota (( \widetilde \Psi^h )'\pro) + \sym \nabla_h \widetilde v^h.
\]
By construction it satisfies
\begin{equation}\begin{aligned}\label{eq:propertieswidetilde}
(\abs { \sym \nabla_h \widetilde \psi^h}^2) \; \text{ is equi-integrable, and }\;\lim_{h\down 0} \norm{ \sym (\nabla_h \psi^h - \nabla_h \widetilde \psi^h)}_{L^2(O^h)} = 0.
\end{aligned}\end{equation}
Furthermore, for any $0 < a < L$ we decompose the domain of integration and obtain
\begin{equation}\begin{aligned}\label{eq:decompdomain}
& \norm{\sym \nabla_h (\widetilde\psi^h - \psi^h_{m_d})}_{L^2(\Omega)}^2 \\
&= \norm{\sym \nabla_h (\widetilde\psi^h - \psi^h_{m_d})}_{L^2((0,a) \times \omega)}^2 + \norm{\sym \nabla_h (\widetilde\psi^h - \psi^h_{m_d})}_{L^2((a,L) \times \omega)}^2.
\end{aligned}\end{equation}
For the second term on the right-hand side we use use the coercivity of $Q^h$ to obtain
\begin{align*}	
\alpha \norm{\sym \nabla_h (\widetilde\psi^h - \psi^h_{m_d})}_{L^2((a,L)\times \omega)}^2
&\leq \frac 1 2 \int_{(a,L)\times \omega} \lina  \nabla_h (\widetilde\psi^h - \psi^h_{m_d}): \nabla_h (\widetilde\psi^h - \psi^h_{m_d}).
\end{align*}	
Let $\rho \in C^\infty([0,L])$ be a cut-off function such that $\rho \geq 0$, $\rho = 0$ on $[0,a/2]$ and $\rho = 1$ on $[a,L]$.
With this we calculate
\begin{align*}	
\alpha \norm{\sym \nabla_h (\widetilde\psi^h - \psi^h_{m_d})}_{L^2((a,L)\times \omega)}^2
&\leq \frac 1 2 \int_\Omega \rho \lina  \nabla_h (\widetilde\psi^h - \psi^h_{m_d}): \nabla_h (\widetilde\psi^h - \psi^h_{m_d}) \\
&=\frac 1 2  \int_\Omega \rho\lina (\iota(m_d) +  \nabla_h \widetilde\psi^h) : \nabla_h (\widetilde\psi^h - \psi^h_{m_d}) \\
&\quad-\frac 1 2  \int_\Omega\rho \lina (\iota(m_d) +  \nabla_h \psi^h_{m_d}) : \nabla_h (\widetilde\psi^h - \psi^h_{m_d}).
\end{align*}
The second term vanishes by virtue of Lemma~\ref{lem:orthabstract}, while for the second one we use the decomposition~\eqref{eq:symGhdecomp}, i.e.,
\[
  \sym G^h = \sym \iota( m_d ) + \sym \nabla_h \psi^h + o^h,
\] to write
\begin{align*}	
  \int_\Omega \rho \lina (\iota(m_d) +  \nabla_h \widetilde\psi^h) : \nabla_h (\widetilde\psi^h - \psi^h_{m_d}) 
  &=\int_\Omega \rho \lina (G^h + o^h): \nabla_h (\widetilde\psi^h - \psi^h_{m_d}) \\
  &\;+ \int_\Omega \rho \lina (\nabla_h(\widetilde \psi^h - \psi^h)) : \nabla_h (\widetilde\psi^h - \psi^h_{m_d}).
 \end{align*}
The sequence $o^h$ converges strongly to $0$ and thus the term containing it vanishes in the limit. By Prop.~\ref{lem:orthabstract} the sequence $\abs{\sym ( \nabla_h \psi^h_{m_d})}^2$ and, by construction, the sequence $\abs{\sym (\nabla_h \widetilde \psi^h)}^2$ are both equi-integrable. Thus by applying Lemma~\ref{lem:orthequi} the first term vanishes. For the second one we decompose $\Omega = O^h \union (\Omega \setminus O^h)$ and estimate with Hölder's inequality
\begin{equation}\begin{aligned}\label{eq:psiapproximationstep}
&\abs*{\int_\Omega  \rho \lina \nabla_h(\widetilde \psi^h -  \psi^h) : \nabla_h (\widetilde\psi^h - \psi^h_{m_d})} \\
    &\leq \beta \abs*{\int_\Omega  \sym\nabla_h(\widetilde \psi^h -  \psi^h) : \sym \nabla_h (\widetilde\psi^h - \psi^h_{m_d})} \\
    &\leq \beta\norm{\sym\nabla_h(\widetilde \psi^h -  \psi^h)}_{L^2(O^h)}\norm{\sym\nabla_h(\widetilde \psi^h -  \psi^h_{m_d})}_{L^2(\Omega)}\\
    & \quad+ \beta\norm{\sym\nabla_h(\widetilde \psi^h -  \psi^h)}_{L^2(\Omega)}\norm{\sym\nabla_h(\widetilde \psi^h -  \psi^h_{m_d})}_{L^2(\Omega \setminus O^h)}.
    \end{aligned}\end{equation}
First note that
\[
\norm{\sym\nabla_h(\widetilde \psi^h -  \psi^h_{m_d})}_{L^2(\Omega)}, \quad \norm{\sym\nabla_h(\widetilde \psi^h -  \psi^h)}_{L^2(\Omega)}
\]
are uniformly bounded in $h$. Furthermore utilizing~\eqref{eq:propertieswidetilde} we obtain that 
\[
\lim_{h\down 0}\norm{\sym\nabla_h(\widetilde \psi^h -  \psi^h)}_{L^2(O^h)} = 0,
\] and thus the first term on the right-hand side vanishes. For the second term on the right-hand side of~\eqref{eq:psiapproximationstep} we apply the equi-integrability of $(\abs{\sym\nabla_h(\widetilde \psi^h -  \psi^h_{m_d})}^2)$ together with $\Lebesgueee{\Omega \setminus  O^h} \to 0$ for $h\down 0$, and obtain that the term vanishes as well.

Returning to~\eqref{eq:decompdomain}, we take a sequence $a = a(h)$ with $a(h) \down 0$ for $h \down 0$ such that
\begin{equation*}\begin{aligned}
 \lim_{h\down 0} \norm{\sym \nabla_h (\widetilde\psi^h - \psi^h_{m_d})}_{L^2((a(h),L) \times \omega)} = 0.
\end{aligned}\end{equation*}
By equi-integrability we also obtain
\begin{equation*}\begin{aligned}
 \lim_{h\down 0} \norm{\sym \nabla_h (\widetilde\psi^h - \psi^h_{m_d})}_{L^2((0, a(h)) \times \omega)} = 0,
\end{aligned}\end{equation*}
and thus
\begin{equation*}\begin{aligned}
 \lim_{h\down 0} \norm{\sym \nabla_h (\widetilde\psi^h - \psi^h_{m_d})}_{L^2(\Omega)}  =0.
\end{aligned}\end{equation*}

\vspace{\baselineskip}

Returning to~\eqref{eq:mainproofapproxEbyA}, we first approximate $\sym \nabla_h \psi^h$ by $\sym \nabla_h \widetilde \psi^h$, and then the latter by $\sym \nabla_h \psi^h_{m_d}$, thus obtaining
\begin{equation}\begin{aligned}\label{eq:bendingidentified}
&\lim_{h\down 0}\int_\Omega E^h: \iota( (A\Phi  - \Phi A)\pro + \Phi' \pro) \td x \\
&\quad= \lim_{h\down 0}\int_\Omega \lina G^h: \iota( (A\Phi  - \Phi A)\pro + \Phi' \pro) \td x\\
&\quad= \lim_{h\down 0}\int_\Omega \lina(\iota(m_d) + \nabla_h \psi^h_{m_d}): \iota( (A\Phi  - \Phi A)\pro + \Phi' \pro) \td x\\
&\quad=\frac { \partial  \mathcal K_{(h)}}{ \partial m}(m_d)[(A\Phi - \Phi A + \Phi')\pro].
\end{aligned}\end{equation}

We combine~\eqref{eq:almostlimit},~\eqref{eq:limitforces} and~\eqref{eq:bendingidentified}, obtaining
\begin{equation}\begin{aligned}\label{eq:limitequationnonopt}
0=\lim_{h\down 0}\int_\Omega \paren*{R^h E^h: \nabla_h \psi - h \widehat g \cdot \partial_1 \psi }\td x   
&=
\frac { \partial  \mathcal K_{(h)}}{ \partial m}(m_d)[(A\Phi - \Phi A + \Phi')\pro] \\
&\quad-\int_0^L \widehat g \cdot (R \Phi e_1)\td x_1.
    \end{aligned}\end{equation}
If 
\begin{equation}\begin{aligned}\label{eq:lastclaim}
\frac { \partial  \mathcal K_{(h)}}{ \partial m}(m_d)[(A\Phi - \Phi A + \Phi')\pro]
=
 \paren*{\frac { \partial}{\partial B}\mathcal K^0_{(h)}}(A) [A\Phi -\Phi A + \Phi']
\end{aligned}\end{equation}
holds, then~\eqref{eq:limitequationnonopt} reads
\begin{equation*}\begin{aligned}
 \paren*{ \frac { \partial}{\partial B}\mathcal K^0_{(h)}}(A) [A\Phi -\Phi A + \Phi']
 = \int_0^L \widehat g \cdot (R\Phi e_1) \td x_1,
\end{aligned}\end{equation*}
and by Lemma~\ref{lem:stationarylimit} this is equivalent to $(\overline y,\overline d_2, \overline d_3)$ being a stationary point of $\energyfunction^0$. 

After replacing both sides in~\eqref{eq:lastclaim} by the more explicit representations~\eqref{eq:mderivativeK} and~\eqref{eq:BderivativeofK0}, we see that it suffices to show that the fixed part $m_d$ is given by $m(A,  b_{\min}(A))$. By definition of $m_d$ in~\eqref{eq:definitionmd} we have $m_d = m(A, p_1)$ where $p_1$ is some $L^2$ function. By the characterization given in~\eqref{eq:widetildebdef} the equality $p_1 =  b_{\min}(A)$ follows, if
\[
\frac { \partial \mathcal K_{(h)}(m(A,\cdot))}{\partial b}(p_1)[\mu] = 0 \quad \text{ for all } \mu \in L^2((0,L)).
\] 
Using~\eqref{eq:limitELnonopt} we see that this is equivalent to
\begin{equation}\begin{aligned}\label{eq:neededequationforoptimality}
\lim_{h\down 0}\int_\Omega \lina(\iota(m_d)) + \sym \nabla_h \psi^h_{m_d}): \iota(\mu e_1)\td x = 0 \quad \text{ for all } \mu \in L^2((0,L)).
\end{aligned}\end{equation}
Similar to before we can replace $\sym \nabla_h \psi^h_{m_d}$ by $\sym \nabla_h \psi^h$. Then we can approximate $\lina G^h$ by $E^h$ with Lemma~\ref{lem:weakvanishingofzeta}, and the statement~\eqref{eq:neededequationforoptimality} is then seen to be equivalent to
\[
\lim_{h\down 0}\int_0^L \overline E^h_{11}\mu \td x = 0 \quad \text{ for all } \mu \in L^2((0,L)),
\]
which now easily follows from~\eqref{eq:constELptwslim}.

\vspace{2\baselineskip}
{\bfseries Acknowledgment:}
The author thanks his PhD advisor Peter Hornung as well as Igor Vel\v ci\'c for many valuable discussions, encouragement and general support.
The author was supported by DFG under Grant~agreement~No.~HO4697/1-1.

\appendix

\section{Appendix}
For convenience of the reader we recall a type of decomposition introduced in~\cite{GrisoRod,GrisoThin}. More precisely the variant proved in~\cite[corollary 2.3, lemma 2.4]{MV16}
\begin{prop}\label{prop:grisoa}
Let $L >0$ and $\Omega = ( 0,L ) \times \omega$, where $\omega$ is an open, connected bounded Lipschitz-domain, which is centered at the origin in the sense of~\eqref{eq:omegacentered}.
Let $(u^h) \subset W^{1,2}(\Omega, \R^3)$ with $\twist(u_2^h, u_3^h) \to 0$ in $L^2((0,L))$,
\[
  \sup_{h > 0} \norm[\big]{\sym { \nabla_h u^h }}_{L^2} < \infty \quad \text{ and } \quad (u_1^h, h u_2^h, h u_3^h) \to 0 \text{ strongly in } L^2(\Omega, \R^3).
\]
Then there exists a constant $C_\omega > 0$, depending only on $\omega$, and sequences $(B^h) \subset W^{1,2}((0,L), \R^{3\times3}_{\skewsym})$, $(\vartheta^h) \subset W^{1,2}(\Omega,\R^3)$ and $(o^h) \subset L^2(\Omega, \R^{3\times 3})$ with
\begin{equation*}\begin{aligned}
 \sym \nabla_h u^h = \sym \iota ( (B^h)' \pro) + \sym \nabla_h \vartheta^h + o^h,
\end{aligned}\end{equation*}
and satisfy the bounds
\[
  \norm{ B^h }_{W^{1,2}} + \norm {\vartheta^h}_{L^2} + \norm{{ \nabla_h \vartheta^h }}_{L^2} \leq C_\omega \norm { \sym \nabla_h u^h}_{L^2}.
\]
Furthermore $B^h, o^h, \vartheta^h \to 0$ strongly in $L^2$. If, in addition, $(\abs { \sym \nabla_h u^h }^2)$ is equi-integrable, then so are $(\abs{(B^h)'}^2)$ and $(\abs { \nabla_h \vartheta^h}^2)$.
\end{prop}
The reverse holds true as well:
\begin{prop}\label{prop:grisob}
Let $L >0$ and $\Omega = ( 0,L ) \times \omega$, where $\omega$ is an open, connected bounded Lipschitz-domain, which is centered at the origin in the sense of~\eqref{eq:omegacentered}.
Let $(B^h) \subset W^{1,2}((0,L), \R^{3\times3}_{\skewsym})$, $(\vartheta^h) \subset W^{1,2}(\Omega,\R^3)$ be sequences with $B^h \to 0$ strongly in $L^2((0,L), \R^{3\times3})$ and $\vartheta^h \to 0$ strongly in $L^2(\Omega,\R^3)$. 
Then there exists $(u^h) \subset W^{1,2}(\Omega, \R^3)$ with $\twist(u_2^h, u_3^h) \to 0$ in $L^2((0,L))$ and
\[
   (u_1^h, h u_2^h, h u_3^h) \to 0 \text{ strongly in } L^2(\Omega, \R^3)
\]
such that
\begin{equation*}\begin{aligned}
 \sym \nabla_h u^h = \sym \iota ( (B^h)' \pro) + \sym \nabla_h \vartheta^h.
\end{aligned}\end{equation*}
\end{prop}

\bibliographystyle{alpha}

%\printbibliography
%\include{main}
%\bibliographystyle{alpha}
%\bibliography{biblio}
%\bibliography{../Promotion/Doktorarbeit/biblio.bib}

\begin{thebibliography}{10}

\bibitem{JB84}
J.~M. Ball.
\newblock Minimizers and the {E}uler-{L}agrange equations.
\newblock In {\em Trends and applications of pure mathematics to mechanics
  ({P}alaiseau, 1983)}, volume 195 of {\em Lecture Notes in Phys.}, pages 1--4.
  Springer, Berlin, 1984.

\bibitem{JB02}
John~M. Ball.
\newblock Some open problems in elasticity.
\newblock In {\em Geometry, mechanics, and dynamics}, pages 3--59. Springer,
  New York, 2002.

\bibitem{bern1692}
James Bernoulli.
\newblock Quadratura curvae, e cujus evolutione describitur inflexae laminae
  curvatura.
\newblock In {\em Die Werke von Jakob Bernoulli}, pages 223--227. Birkh\"auser,
  1692.
\newblock Med. CLXX; Ref. UB: L Ia 3, p 211--212.

\bibitem{BFF00}
A.~Braides, I.~Fonseca, and G.~Francfort.
\newblock 3{D}-2{D} asymptotic analysis for inhomogeneous thin films.
\newblock {\em Indiana Univ. Math. J.}, 49(4):1367--1404, 2000.

\bibitem{BVP17}
Mario Bukal, Matth{\"a}us Pawelczyk, and Igor Vel{\v c}i{\'c}.
\newblock Derivation of homogenized {E}uler--{L}agrange equations for von
  {K}{\'a}rm\'an rods.
\newblock {\em J. Differential Equations}, 262(11):5565--5605, 2017.

\bibitem{DM12}
Elisa Davoli and Maria~Giovanna Mora.
\newblock Convergence of equilibria of thin elastic rods under physical growth
  conditions for the energy density.
\newblock {\em Proc. Roy. Soc. Edinburgh Sect. A}, 142(3):501--524, 2012.

\bibitem{euler1744}
Leonhard Euler.
\newblock {\em Methodus inveniendi lineas curvas maximi minimive proprietate
  gaudentes, sive solutio problematis isoperimetrici lattissimo sensu accepti},
  chapter Additamentum 1.
\newblock eulerarchive.org E065, 1744.

\bibitem{FJM02}
Gero Friesecke, Richard~D. James, and Stefan M{\"u}ller.
\newblock A theorem on geometric rigidity and the derivation of nonlinear plate
  theory from three-dimensional elasticity.
\newblock {\em Comm. Pure Appl. Math.}, 55(11):1461--1506, 2002.

\bibitem{FJM06}
Gero Friesecke, Richard~D. James, and Stefan M{\"u}ller.
\newblock A hierarchy of plate models derived from nonlinear elasticity by
  gamma-convergence.
\newblock {\em Arch. Ration. Mech. Anal.}, 180(2):183--236, 2006.

\bibitem{GrisoRod}
Georges Griso.
\newblock Asymptotic behavior of structures made of curved rods.
\newblock {\em Anal. Appl. (Singap.)}, 6(1):11--22, 2008.

\bibitem{GrisoThin}
Georges Griso.
\newblock Decompositions of displacements of thin structures.
\newblock {\em J. Math. Pures Appl. (9)}, 89(2):199--223, 2008.

\bibitem{LDR96}
Herv\'e Le~Dret and Annie Raoult.
\newblock The nonlinear membrane model as variational limit of nonlinear
  three-dimensional elasticity.
\newblock {\em J. Math. Pures Appl. (9)}, 74(6):549--578, 1995.

\bibitem{levien08}
Raph Levien.
\newblock The elastica: a mathematical history.
\newblock Technical Report UCB/EECS-2008-103, EECS Department, University of
  California, Berkeley, Aug 2008.

\bibitem{MV16}
Maroje Marohni{\'c} and Igor Vel{\v{c}}i{\'c}.
\newblock Non-periodic homogenization of bending-torsion theory for
  inextensible rods from 3{D} elasticity.
\newblock {\em Ann. Mat. Pura Appl. (4)}, 195(4):1055--1079, 2016.

\bibitem{MM08}
M.~G. Mora and S.~M{\"u}ller.
\newblock Convergence of equilibria of three-dimensional thin elastic beams.
\newblock {\em Proc. Roy. Soc. Edinburgh Sect. A}, 138(4):873--896, 2008.

\bibitem{MMS06}
M.~G. Mora, S.~M\"uller, and M.~G. Schultz.
\newblock Convergence of equilibria of planar thin elastic beams.
\newblock {\em Indiana Univ. Math. J.}, 56(5):2413--2438, 2007.

\bibitem{MM02}
Maria~Giovanna Mora and Stefan M{\"u}ller.
\newblock Derivation of the nonlinear bending-torsion theory for inextensible
  rods by {$\Gamma$}-convergence.
\newblock {\em Calc. Var. Partial Differential Equations}, 18(3):287--305,
  2003.

\bibitem{N10}
Stefan Neukamm.
\newblock {\em Homogenization, linearization and dimension reduction in
  elasticity with variational methods}.
\newblock PhD thesis, Technische Universit{\"a}t M{\"u}nchen, 2010.

\bibitem{VvK17}
Igor Vel{\v c}i\'c.
\newblock On the general homogenization of von {K}\'arm\'an plate equations
  from three-dimensional nonlinear elasticity.
\newblock {\em Anal. Appl. (Singap.)}, 15(1):1--49, 2017.

\end{thebibliography}
 \end{document}